\documentclass[a4paper,12pt,reqno]{amsart}

\usepackage[T1]{fontenc}
\usepackage[utf8]{inputenc}
\usepackage{lmodern}
\usepackage[english]{babel}

\usepackage[%
	left=2.5cm,       
	right=2.5cm,      
	top=3.5cm,        
	bottom=3.5cm,     
	heightrounded,    
	bindingoffset=0mm 
]{geometry}

\usepackage{amsmath}
\usepackage{amssymb}
\usepackage{mathtools}
\usepackage{mathrsfs}
\usepackage[normalem]{ulem}

\numberwithin{equation}{section}

\usepackage{hyperref} 
\usepackage[noabbrev,capitalize]{cleveref}

\usepackage{bookmark}

\usepackage[initials,
]{amsrefs}

\theoremstyle{plain}
	\newtheorem{theorem}{Theorem}[section]
	\newtheorem{lemma}[theorem]{Lemma}
	\newtheorem{proposition}[theorem]{Proposition}
	\newtheorem{corollary}[theorem]{Corollary}

\theoremstyle{definition}
	\newtheorem{definition}[theorem]{Definition}
	\newtheorem{example}[theorem]{Example}
	\newtheorem{remark}[theorem]{Remark}
	\newtheorem{open.problem}[theorem]{Open Problem}

\newcommand{\N}{\mathbb{N}}
\newcommand{\R}{\mathbb{R}}

\newcommand{\eps}{\varepsilon}

\usepackage[initials]{amsrefs}

\usepackage{tikz}
\usepackage{graphicx}
\usepackage{xcolor}
\usepackage[font=small]{caption}

\newcommand{\closure}[2][3]{%
  {}\mkern#1mu\overline{\mkern-#1mu#2}}
  
\usepackage{enumerate}

\usepackage{url}

\newcommand{\de}{\partial}

\newcommand{\weakto}{\rightharpoonup}

\renewcommand{\phi}{\varphi}
\renewcommand{\rho}{\varrho}
\renewcommand{\theta}{\vartheta}
\newcommand{\M}{\mathscr{M}}

\DeclareMathOperator{\supp}{supp}
\DeclareMathOperator{\sgn}{sgn}

\DeclareMathOperator{\diverg}{div}

\DeclareMathOperator{\Lip}{Lip}
\DeclareMathOperator{\loc}{loc}

\DeclareMathOperator{\Capa}{Cap}

\DeclarePairedDelimiter{\set}{\{}{\}}
\DeclarePairedDelimiter{\abs}{|}{|}

\mathchardef\ordinarycolon\mathcode`\:
\mathcode`\:=\string"8000
\begingroup \catcode`\:=\active
  \gdef:{\mathrel{\mathop\ordinarycolon}}
\endgroup

\newcommand{\mres}{\mathbin{\vrule height 1.6ex depth 0pt width
0.13ex\vrule height 0.13ex depth 0pt width 1.3ex}}

\def\Xint#1{\mathchoice
{\XXint\displaystyle\textstyle{#1}}%
{\XXint\textstyle\scriptstyle{#1}}%
{\XXint\scriptstyle\scriptscriptstyle{#1}}%
{\XXint\scriptscriptstyle\scriptscriptstyle{#1}}%
\!\int}
\def\XXint#1#2#3{{\setbox0=\hbox{$#1{#2#3}{\int}$ }
\vcenter{\hbox{$#2#3$ }}\kern-.6\wd0}}

\def\aint{\Xint-}

\newcommand{\Haus}[1]{\mathscr{H}^{#1}} % Misura di Hausdorff
\newcommand{\Leb}[1]{\mathscr{L}^{#1}} % Misura di Lebesgue
\renewcommand{\div}{\mathrm{div}} %divergence
\newcommand{\redb}{\mathscr{F}} %reduced boundary 
\newcommand{\res}{\mathop{\hbox{\vrule height 7pt width .5pt depth 0pt
\vrule height .5pt width 6pt depth 0pt}}\nolimits}

\allowdisplaybreaks

\begin{document}

\title[The fractional variation and the precise representative]{The fractional variation and the precise representative of $BV^{\alpha,p}$ functions}

\author[G.~E.~Comi]{Giovanni E. Comi}
\address[G.~E.~Comi]{Dipartimento di Matematica, Università di Pisa, Largo Bruno Pontecorvo 5, 56127 Pisa, Italy}
\email{giovanni.comi@dm.unipi.it}

\author[D. Spector]{Daniel Spector}
\address[D. Spector]{Department of Mathematics, National Taiwan Normal University, No. 88, Section 4, Tingzhou Road, Wenshan District, Taipei City, Taiwan 116, R.O.C.; Okinawa Institute of Science and Technology Graduate University, Nonlinear Analysis Unit, 1919-1 Tancha, Onna-son, Kunigami-gun, Okinawa, Japan 904-0495}
\email{spectda@protonmail.com}

\author[G. Stefani]{Giorgio Stefani}
\address[G. Stefani]{Department Mathematik und Informatik, Universit\"at Basel, Spiegelgasse 1, CH-4051 Basel, Switzerland}
\email{giorgio.stefani.math@gmail.com}

\thanks{\textit{Acknowledgments}. 
The first author is a member of INdAM--GNAMPA and 
is partially supported by the PRIN 2017 Project \textit{Variational methods for stationary and evolution problems with singularities and interfaces} (n.\ prot.\ 2017BTM7SN). 
The second author is supported by the Taiwan Ministry of Science and Technology under research grant number 110-2115-M-003-020-MY3 and the Taiwan Ministry of Education under the Yushan Fellow Program.
Part of this work was undertaken while the second author was visiting the National Center for Theoretical Sciences in Taiwan. He would like to thank the NCTS for its support and warm hospitality during the visit. 
The third author is a member of INdAM--GNAMPA and 
is partially supported by the ERC Starting Grant 676675 FLIRT -- \textit{Fluid Flows and Irregular Transport} and by the INdAM--GNAMPA Project 2020 \textit{Problemi isoperimetrici con anisotropie} (n.\ prot.\ U-UFMBAZ-2020-000798 15-04-2020).
}

\keywords{Fractional gradient, fractional divergence, fractional variation, Hausdorff measure, fractional capacity, precise representative}

\subjclass[2020]{Primary 46E35. Secondary 28A12}

\date{\today}

\begin{abstract}
We continue the study of the fractional variation following the distributional approach developed in the previous works~\cites{BCCS20,CS19,CS19-2}. 
We provide a general analysis of the distributional space $BV^{\alpha,p}(\R^n)$ of $L^p$ functions, with $p\in[1,+\infty]$, possessing finite fractional variation of  order $\alpha\in(0,1)$.
Our two main results deal with the absolute continuity property of the fractional variation with respect to the Hausdorff measure and the existence of the precise representative of a $BV^{\alpha,p}$ function. 
\end{abstract}

\maketitle

\tableofcontents

\section{Introduction}

\subsection{The fractional variation}

For a parameter $\alpha\in(0,1)$ and an exponent $p\in[1,+\infty]$, the \emph{space of $L^p$ functions with bounded fractional variation} is
\begin{equation}
\label{eqi:def_BV_alpha_p_space}
BV^{\alpha,p}(\R^n)
=
\set*{f\in L^p(\R^n) : |D^\alpha f|(\R^n)<+\infty},
\end{equation}
where
\begin{equation}
\label{eqi:def_frac_var_tot}
|D^\alpha f|(\R^n)
=
\sup\set*{\int_{\R^n} f\,\div^\alpha\phi\,dx : \phi\in C^\infty_c(\R^n;\R^n),\ \|\phi\|_{L^\infty(\R^n;\,\R^n)}\le1}
\end{equation}
is the (total) \emph{fractional variation} of the function $f\in L^p(\R^n)$.
Here and in the following, for sufficiently smooth functions and vector-fields, we let
\begin{equation*}\nabla^\alpha f(x)
=
\mu_{n,\alpha}
\int_{\R^n}\frac{(y-x)(f(y)-f(x))}{|y-x|^{n+\alpha+1}}\,dy,
\quad
x\in\R^n,	
\end{equation*}
and
\begin{equation*}
\div^\alpha\phi(x)
=
\mu_{n,\alpha}
\int_{\R^n}\frac{(y-x)\cdot(\phi(y)-\phi(x))}{|y-x|^{n+\alpha+1}}\,dy,
\quad
x\in\R^n,	
\end{equation*}
be the \emph{fractional gradient} and the \emph{fractional divergence} operators respectively, where $\mu_{n,\alpha}$ is a suitable renormalizing constant depending on $n$ and $\alpha$ only.
The above fractional operators are \emph{dual}, in the sense that
\begin{equation}
\label{eqi:ibp}
\int_{\R^n}f\,\div^\alpha\phi \,dx=-\int_{\R^n}\phi\cdot\nabla^\alpha f\,dx.
\end{equation}

%The fractional operators $\nabla^\alpha$ and $\div^\alpha$, together with the integration-by-parts formula~\eqref{eqi:ibp}, open the doors for the development of a new (distributional) theory for functions possessing some regularity of fractional order. 

The fractional variation was considered by the first and the third authors in the work~\cite{CS19} in the \emph{geometric} framework $p=1$, also in relation with the naturally associated notion of \emph{fractional Caccioppoli perimeter}.
The fractional variation of an $L^p$ function for an arbitrary exponent $p\in[1,+\infty]$ was then studied by the same authors in the subsequent paper~\cite{CS19-2}, in connection with some embedding-type results arising from some optimal inequalities proved by the second author~\cites{S19,S20-An}.

Since the first appearance of the fractional gradient~\cite{H59}, the literature around $\nabla^\alpha$ and $\div^\alpha$ has been rapidly growing in various directions, such as the study of PDEs~\cites{SS15,SS18,SSS15,SSS18} and of functionals~\cites{BCM20,BCM21,KS21} involving these fractional operators, the discovery of new optimal embedding estimates~\cites{SSS17,S19,S20-An} and the development of a distributional and asymptotic analysis in this fractional framework~\cites{BCCS20,CS19,CS19-2,Sil19}.
We also refer the reader to the survey~\cite{S20-New} and to the monograph~\cite{P16}.

At the present stage of the theory, the fine properties of functions having finite fractional variation are not completely understood and, to our knowledge, only some results~\cite{CS19} in the \emph{geometric regime} $p=1$ are available in the literature.

Besides providing a general treatment of the space $BV^{\alpha,p}(\R^n)$, in the present paper we aim to develop the existing theory in this direction.
On the one side, we study the relation between the fractional variation and the Hausdorff measure.
On the other side, we establish the existence of the precise representative of a $BV^{\alpha,p}$ function.

\subsection{The Hausdorff dimension of the fractional variation}

The natural idea behind the definition of the space $BV^{\alpha,p}(\R^n)$ is that a function $f\in L^p(\R^n)$ belongs to $BV^{\alpha,p}(\R^n)$ if and only if there exists a finite vector-valued Radon measure $D^\alpha f\in\M(\R^n;\R^n)$ such that 
\begin{equation*}
%\label{eqi:def_frac_var_p}
\int_{\R^n} f\,\div^\alpha\phi\,dx
=
-
\int_{\R^n}\phi\cdot d D^\alpha f
\end{equation*}
for all $\phi\in C^\infty_c(\R^n;\R^n)$, generalizing the integration-by-parts formula~\eqref{eqi:ibp}. 

In the classical integer case $\alpha=1$, the variation of a function  $f\in BV(\R^n)$ is known to satisfy
\begin{equation}
\label{eqi:var_haus}
|Df|\ll\Haus{n-1},	
\end{equation}
where $\Haus{s}$ is the $s$-dimensional Hausdorff measure. 
If $f=\chi_E$ for some measurable set $E\subset\R^n$, then it actually holds that 
\begin{equation}
\label{eqi:redb}
|D\chi_E|
=
\Haus{n-1}\mres\redb E,
\end{equation} 
where $\redb E$ is the De Giorgi \emph{reduced boundary} of $E$, see the monographs~\cites{AFP00,M12}.
 
Roughly speaking,  formulas~\eqref{eqi:var_haus} and~\eqref{eqi:redb} mean that the variation measure of a $BV$ function in~$\R^n$ lives on sets with Hausdorff dimension~$n-1$ at least.
By the analogy between the integer and the fractional settings, one may expect that a similar phenomenon should occur also for the fractional variation of order $\alpha\in(0,1)$ on a set of  Hausdorff dimension~$n-\alpha$ at least.
In~\cite{CS19}*{Corollary~5.4}, the first and the third authors confirmed this parallelism by showing that, for a measurable set $E\subset\R^n$ such that $\chi_E \in BV^{\alpha}(\R^n)$ (or, more generally, for any measurable set having locally finite \emph{fractional Caccioppoli perimeter}, see~\cite{CS19}*{Definition~4.1}), it holds that
\begin{equation}
\label{eqi:redb_frac}
|D^\alpha\chi_E|
\le 
c_{n,\alpha}\,
\Haus{n-\alpha}\mres\redb^\alpha E,
\end{equation} 
where $c_{n,\alpha}>0$ depends on $n$ and $\alpha$ only and $\redb^\alpha E$ is the fractional analogue of the De Giorgi \emph{reduced boundary}~\eqref{eqi:redb}, the so-called \emph{fractional reduced boundary} of $E$, see~\cite{CS19}*{Definition~4.7}.
However, as shown in~\cite{CS19}*{Lemma~3.28} by the same authors, if $f\in BV^\alpha(\R^n)$ then the function $u=I_{1-\alpha}f$ (where $I_s$ is the \emph{Riesz potential} of order $s\in(0,n)$, see below for the precise definition) does satisfy $|Du|(\R^n)<+\infty$, with
\begin{equation}
\label{eqi:frac_var_Riesz}
Du=D^\alpha f
\quad
\text{in}\ \M(\R^n;\R^n).
\end{equation}
In particular, by combining~\eqref{eqi:var_haus} with the above~\eqref{eqi:frac_var_Riesz}, we immediately get that 
\begin{equation}
\label{eqi:frac_var_haus}
|D^\alpha f|
\ll
\Haus{n-1}
\end{equation} 
for all $f\in BV^\alpha(\R^n)$, thus ruling out the existence of a \emph{coarea formula} in this fractional setting, see~\cite{CS19}*{Corollary~5.6}.

Equations~\eqref{eqi:redb_frac} and~\eqref{eqi:frac_var_haus} illustrate the richness arising from the innocent-looking definition~\eqref{eqi:def_frac_var_tot} and lead to the idea that the behavior of the fractional variation of a function $f\in L^p(\R^n)$ may depend on its integrability exponent $p\in[1,+\infty]$. 
Our first main result provides a rigorous formulation of this intuitive idea and can be stated as follows.

\begin{theorem}[Absolute continuity properties of the fractional variation]
\label{res:abs_frac_var}
Let $\alpha\in(0,1)$, $p\in[1,+\infty]$ and assume that $f\in BV^{\alpha,p}(\R^n)$. We have the following cases:

\begin{enumerate}[(i)]

\item 
\label{item:abs_frac_var_subcritical}
if $p\in\left[1,\frac n{1-\alpha}\right)$, then
$|D^\alpha f|\ll\Haus{n-1}$;

\item 
\label{item:abs_frac_var_supercritical}
if $p\in\left[\frac n{1-\alpha},+\infty\right]$, then
$|D^\alpha f|\ll\Haus{n-\alpha-\frac np}$.
\end{enumerate}
\end{theorem}

As shown by  \cref{res:abs_frac_var}, the fractional variation in the \emph{subcritical regime} $p<\frac n{1-\alpha}$ is comparable with the Hausdorff measure of dimension~$n-1$, in accordance with~\eqref{eqi:frac_var_haus}.
In fact, we can actually prove a deeper property, in analogy with the relation~\eqref{eqi:frac_var_Riesz}.
Precisely, the Riesz potential operator
\begin{equation*}
I_{1-\alpha}
\colon
BV^{\alpha,p}(\R^n)
\to
BV^{1,\frac{np}{n-(1-\alpha)p}}(\R^n)
\end{equation*}
is continuous whenever $p<\frac n{1-\alpha}$ (see \cref{res:ext_lemma328}\eqref{item:Riesz_of_BV_alpha_p} below for the detailed statement), from which item~\eqref{item:abs_frac_var_subcritical} in \cref{res:abs_frac_var} immediately follows.
Here and in the following, for any $p\in[1,+\infty]$, we let
\begin{equation*}
BV^{1,p}(\R^n) = \set*{f \in L^p(\R^n) : |D f|(\R^n) < + \infty}
\end{equation*}
be the space of $L^p$ functions having finite variation, extending the definition  in~\eqref{eqi:def_BV_alpha_p_space} to the integer case $\alpha=1$.

In the \emph{supercritical regime} $p\ge\frac n{1-\alpha}$ instead, the fractional variation is comparable with the Hausdorff measure of dimension~$n-\alpha-\frac np$, thus recovering~\eqref{eqi:redb_frac} in the case $p=+\infty$. 
The proof of item~\eqref{item:abs_frac_var_supercritical} of \cref{res:abs_frac_var} is more delicate and requires a finer analysis. 
The overall idea is to adapt the strategy developed in~\cite{CS19}*{Section~5} for sets with (locally) finite \emph{fractional Caccioppoli perimeter}  to the present more general $L^p$ framework.
The key role in this approach is played by the following \emph{decay estimate} for the fractional variation of a function $f\in BV^{\alpha,p}(\R^n)$ with $p\ge\frac n{n-\alpha}$,
\begin{equation}
\label{eqi:decay_frac}
|D^{\alpha} f|(B_r(x)) 
\le 
c_{n,\alpha,p} 
\|f\|_{L^p(\R^n)}\,
r^{n-\alpha-\frac{n}{p}}, 	
\end{equation}
valid for $|D^\alpha f|$-a.e.\ $x\in\R^n$ and all $r>0$ sufficiently small, where $c_{n,\alpha,p}>0$ is a constant depending on~$n$, $\alpha$, and~$p$ only (see \cref{res:decay_estimate_p} below for the precise statement). The validity of \eqref{eqi:decay_frac} is suggested by the following heuristic argument, valid for all $f \in BV^{\alpha, p}(\R^n)$ such that 
\begin{equation}
\label{eqi:heu_ass}
(D^{\alpha} f)_j \ge 0
\quad
\text{for all}\ 
j \in \set{1, \dots, n}.
\end{equation}
If $\phi \in C^{\infty}_{c}(B_2)$ is such that $\phi \ge 0$ and $\phi \equiv 1$ on $B_1$, then
\begin{equation*} 
(D^{\alpha} f)_j(B_r(x)) \le \int_{\R^n} \phi\left (\frac{y - x}{r} \right ) \, d (D^{\alpha} f)_j(y) 
= 
- r^{- \alpha}  \int_{\R^n} f(y) \,( \nabla^{\alpha} \phi)_j \left (\frac{y - x}{r} \right ) \, dy, \end{equation*}
thanks to \eqref{eqi:ibp} and the $\alpha$-homogeneity of the fractional gradient (\cite{Sil19}*{Theorem 4.3}), so that
\begin{align*} (D^{\alpha} f)_j(B_r(x)) 
&\le 
\|f\|_{L^{p}(\R^n)} \left ( \int_{\R^n} |\nabla^{\alpha} \phi(y)|^{\frac{p}{p-1}} r^{n} \, dy \right )^{1 - \frac{1}{p}} r^{- \alpha} 
\\
& = 
\|f\|_{L^{p}(\R^n)} \|\nabla^{\alpha} \phi \|_{L^{\frac{p}{p-1}}(\R^n; \R^{n})} r^{n - \alpha - \frac{n}{p}}, \end{align*}
which gives~\eqref{eqi:decay_frac}. 
Without~\eqref{eqi:heu_ass}, the decay estimate~\eqref{eqi:decay_frac} is a consequence of some new integrability properties in Lorentz spaces of the fractional gradient and of an integration-by-parts formula of $BV^{\alpha,p}$ functions on balls which may be of some independent interest (see \cref{res:optimal_estimate} and \cref{res:int_by_parts_E_ball}, respectively). 

We note that \cref{res:abs_frac_var} still holds even in the limit as $\alpha \to 1^-$. 
Indeed, for all $p \in [1, +\infty]$ and $f \in BV^{1,p}(\R^n)$ we get that $|D f|\ll\Haus{n-1}$, since point \eqref{item:abs_frac_var_subcritical} now applies to all $p \in [1, +\infty)$, while point \eqref{item:abs_frac_var_supercritical} refers only to $p=+\infty$, for which we have $n-1-\frac np = n - 1$. 
This is in fact a well-known result for functions in $BV_{\rm loc}(\R^n)$, see \cite{AFP00}*{Lemma~3.76} for instance. 
On the contrary,  \cref{res:abs_frac_var} is not optimal in the limit as $\alpha \to 0^+$.
Indeed, in virtue of \cite{BCCS20}*{Theorem~3.3 and Remark~A.3}, if $p \in [1, +\infty)$ and $f \in BV^{0,p}(\R^n)$, then $|D^0 f| \ll \Leb{n}$ (where the space $BV^{0,p}(\R^n)$ is defined as in \eqref{eqi:def_BV_alpha_p_space} with $\alpha = 0$, see~\cite{BCCS20} for a more detailed presentation).

\subsection{The precise representative of a \texorpdfstring{$BV^{\alpha,p}$}{BVˆ(alpha,p)} function}

Formulas~\eqref{eqi:var_haus} and~\eqref{eqi:redb} suggest that the set of  discontinuity points (in the   measure-theoretical sense) of a $BV$ function should have Hausdorff dimension~$n-1$.
In more precise terms, if $f\in BV(\R^n)$, then the limit 
\begin{equation}
\label{eqi:def_prec_rep}
f^\star(x)
=
\lim_{r\to0^+}
\aint_{B_r(x)}
f(y)\,dy
\end{equation}
exists for $\Haus{n-1}$-a.e.\ $x\in\R^n$.
In fact, the limit in~\eqref{eqi:def_prec_rep} can be strengthened as
\begin{equation*}
\lim_{r\to0^+}
\aint_{B_r(x)}
|f(y)-f^\star(x)|^{\frac n{n-1}}\,dy=0
\end{equation*}
for $\Haus{n-1}$-a.e.\ $x\in\R^n\setminus J_f$, see~\cite{EG15}*{Section~5.9} for example, where $J_f\subset\R^n$ is the so-called \emph{jump set} of the function $f\in BV(\R^n)$ (if $f\in W^{1,1}(\R^n)$, then $J_f$ is empty). 

The function $f^\star$ defined by~\eqref{eqi:def_prec_rep} is the so-called \emph{precise representative} of the function~$f$ (by convention, we set $f^*(x)=0$ if the limit in~\eqref{eqi:def_prec_rep} does not exist).
The well-posedness of the precise representative~\eqref{eqi:def_prec_rep} of a $BV^{\alpha,p}$ function is not known at the present moment.
Our second main result moves in this direction and can be briefly stated as follows (for a more precise statement, we refer the reader to \cref{res:quasicont_any_p} below).

\begin{theorem}[The precise representative of a $BV^{\alpha,p}$ function]
\label{resi:frac_rep}
Let $\alpha\in(0,1)$, $p\in[1,+\infty]$ and $\eps>0$.
If $f\in BV^{\alpha,p}(\R^n)$, then the limit $f^\star(x)$ 
exists for $\Haus{n - \alpha + \eps}$-a.e.\ $x\in\R$.
Moreover, for any such point $x\in\R^n$, it holds that
\begin{equation*}
\lim_{r\to0^+}
\aint_{B_r(x)}
|f(y)-f^\star(x)|^q\,dy=0
\end{equation*}
for any $q\in\left[1,\bar q_\eps\right]$, where $\bar q_\eps\in\left[1,\frac n{n-\alpha}\right)$ is such that $\lim\limits_{\eps\to0^+}\bar q_\eps=\frac n{n-\alpha}$.   
\end{theorem}

The idea behind the proof of \cref{resi:frac_rep} is very simple and relies on three ingredients naturally arising from our general investigation of the $BV^{\alpha,p}(\R^n)$ space. 
First, we show that $C^\infty_c$ functions are dense in energy in $BV^{\alpha,p}(\R^n)$ provided that $p\in\left[1,\frac n{n-\alpha}\right)$, extending the approximation~\cite{CS19}*{Theorem~3.8} already proved by the first and the third author in the \emph{geometric regime}~$p=1$.
Second, by combining this approximation with an optimal embedding inequality~\cite{S20-An} due to the second author, we establish a fractional analogue of the Gagliardo--Nirenberg--Sobolev inequality, that is, $BV^{\alpha,p}(\R^n)\subset L^{\frac n{n-\alpha}}(\R^n)$ with continuous inclusion.
Third, we exploit this fractional embedding inequality to prove the continuous inclusion of $BV^{\alpha,p}(\R^n)$ into some \emph{Bessel potential space} of suitable fractional order.   
At this point, the existence of the precise representative of a $BV^{\alpha,p}$ function for $p<\frac n{n-\alpha}$ can be inferred from the known theory of \emph{Bessel potential spaces}, see~\cite{AH96}*{Section~6.1} for example.
The remaining exponents $p\ge\frac n{n-\alpha}$ can be recovered from the previous analysis by a simple cut-off argument that may be of some separate interest (see \cref{res:localization} for the detailed statement). 

\subsection{Future developments}

Generally speaking, the precise representative of a function turns out to be the correct object when dealing with the product between the function itself and a sufficiently well-behaved measure.

For example, the precise representative allows to state the general Leibniz rule for the product of two $BV$ functions.
Precisely, if $f,g\in BV(\R^n)\cap L^\infty(\R^n)$, then $fg\in BV(\R^n)$ with
\begin{equation}
\label{eqi:leibniz_bv}
D(fg)
=
g^\star\,Df
+
f^\star\,Dg
\quad
\text{in}\
\M(\R^n;\R^n).
\end{equation}
Note that the two products appearing in right-hand side of~\eqref{eqi:leibniz_bv} are well posed thanks to the combination of the absolute continuity property of the variation~\eqref{eqi:var_haus} and the existence of the precise representative~\eqref{eqi:def_prec_rep}.

With \cref{res:abs_frac_var} and \cref{resi:frac_rep} at hand, the analysis developed in the present work naturally leads to study the interactions between the fractional variation measure and the precise representative of $BV^{\alpha,p}$ functions, aiming at a more general formulation of the Leibniz rule and of the Gauss--Green formula in this fractional setting. 
These results are the main topic of the subsequent paper~\cite{CS21}.

\subsection{Organization of the paper}

The paper is organized as follows.

In \cref{sec:preliminaries}, we quickly set up the notation used throughout the entire work and recall the elementary features of the fractional operators involved.
  
In \cref{sec:BV_alpha_p_space}, we carry out the general analysis of the $BV^{\alpha,p}(\R^n)$ space.
On the one side, we deal with the approximation in energy by smooth functions and the consequent embedding theorems in Lebesgue and Bessel potential spaces, preparing the ground for the proof of \cref{resi:frac_rep}.
On the other side, we treat some integration-by-parts formulas of $BV^{\alpha,p}$ functions against rough test vector-fields and on balls, developing the tools needed for the proof of the decay estimate~\eqref{eqi:decay_frac} and thus of \cref{res:abs_frac_var}. 

In \cref{sec:abs_frac_var}, we prove our first main result \cref{res:abs_frac_var}.
We divide the proof into two parts, dealing with the \emph{subcritical regime}~\eqref{item:abs_frac_var_subcritical} and the \emph{supercritical regime}~\eqref{item:abs_frac_var_supercritical} separately, see \cref{res:ext_lemma328}\eqref{item:Riesz_of_BV_alpha_p} and \cref{res:abs_continuity_p} respectively.
At the end of this section, we provide two examples to show the sharpness of our result in the one-dimensional case~$n=1$. 

In \cref{sec:capa_prec_rep}, after having recalled some known properties of the fractional capacity in Bessel potential spaces  and having proved a localization lemma for $BV^{\alpha,p}$ functions,  we end our paper with the proof of our second main result \cref{resi:frac_rep}.

\section{Preliminaries}
\label{sec:preliminaries}

\subsection{General notation}

We start with a brief description of the main notation used in this paper. In order to keep the exposition the most reader-friendly as possible, we retain the same notation adopted in the previous works~\cites{CS19,CS19-2,BCCS20}.

Given an open set $\Omega\subset\R^n$, we say that a set $E$ is compactly contained in $\Omega$, and we write \mbox{$E\Subset\Omega$}, if the $\closure{E}$ is compact and contained in $\Omega$.
We let $\Leb{n}$ and $\Haus{\alpha}$ be the $n$-dimensional Lebesgue measure and the $\alpha$-dimensional Hausdorff measure on $\R^n$, respectively, with $\alpha\in[0,n]$. 
Unless otherwise stated, a measurable set is a $\Leb{n}$-measurable set. 
We also use the notation $|E|=\Leb{n}(E)$. 
All functions we consider in this paper are Lebesgue measurable, unless otherwise stated. 
We denote by $B_r(x)$ the standard open Euclidean ball with center $x\in\R^n$ and radius $r>0$. 
We let $B_r=B_r(0)$. 
For all $\beta > 0$, we set $\omega_{\beta} =\pi^{\frac{\beta}{2}}/\Gamma\left(\frac{\beta+2}{2}\right)$, where $\Gamma$ is Euler's \emph{Gamma function}, and we recall that $|B_1| = \omega_{n}$ and $\Haus{n-1}(\partial B_{1}) = n \omega_n$.

For $k \in \N_{0} \cup \set{+ \infty}$ and $m \in \N$, we let $C^{k}_{c}(\Omega ; \R^{m})$ and $\Lip_c(\Omega; \R^{m})$ be the spaces of $C^{k}$-regular and, respectively, Lipschitz-regular, $m$-vector-valued functions defined on~$\R^n$ with compact support in the open set~$\Omega\subset\R^n$.

For $m\in\N$, the total variation on~$\Omega$ of the $m$-vector-valued Radon measure $\mu$ is defined as
\begin{equation*}
|\mu|(\Omega)
=
\sup\set*{\int_\Omega\phi\cdot d\mu : \phi\in C^\infty_c(\Omega;\R^m),\ \|\phi\|_{L^\infty(\Omega;\,\R^m)}\le1}.
\end{equation*}
We thus let $\M(\Omega;\R^m)$ be the space of $m$-vector-valued Radon measure  with finite total variation on $\Omega$.
We say that $(\mu_k)_{k\in\N}\subset\M(\Omega;\R^m)$ \emph{weakly converges} to $\mu\in\M (\Omega;\R^m)$, and we write $\mu_k\weakto\mu$ in $\M (\Omega;\R^m)$ as $k\to+\infty$, if 
\begin{equation}\label{eq:def_weak_conv_meas}
\lim_{k\to+\infty}\int_\Omega\phi\cdot d\mu_k=\int_\Omega\phi\cdot d\mu
\end{equation} 
for all $\phi\in C_c^0(\Omega;\R^m)$. Note that we make a little abuse of terminology, since the limit in~\eqref{eq:def_weak_conv_meas} actually defines the \emph{weak*-convergence} in~$\M (\Omega;\R^m)$. 

For any exponent $p\in[1,+\infty]$, we let $L^p(\Omega;\R^m)$ be the space of $m$-vector-valued Lebesgue $p$-integrable functions on~$\Omega$.

We let
\begin{equation*}
W^{1,p}(\Omega;\R^m)
=
\set*{u\in L^p(\Omega;\R^m) : [u]_{W^{1,p}(\Omega;\,\R^m)}=\|\nabla u\|_{L^p(\Omega;\, \R^{nm})}<+\infty}
\end{equation*}
be the space of $m$-vector-valued Sobolev functions on~$\Omega$, see for instance~\cite{Leoni17}*{Chapter~11} for its precise definition and main properties. We also let
\begin{equation*}
w^{1,p}(\Omega;\R^m)
=
\set*{u\in L^p_{\loc}(\Omega;\R^m) : [u]_{W^{1,p}(\Omega;\,\R^m)}<+\infty}.
\end{equation*}
We let
\begin{equation*}
BV(\Omega;\R^m)
=
\set*{u\in L^1(\Omega;\R^m) : [u]_{BV(\Omega;\,\R^m)}=|Du|(\Omega)<+\infty}
\end{equation*}
be the space of $m$-vector-valued functions of bounded variation on~$\Omega$, see for instance~\cite{AFP00}*{Chapter~3} or~\cite{EG15}*{Chapter~5} for its precise definition and main properties. We also let 
\begin{equation*}
bv(\Omega;\R^m)=\set*{u\in L^1_{\loc}(\Omega;\R^m) : [u]_{BV(\Omega;\,\R^m)}<+\infty}.
\end{equation*}
For $\alpha\in(0,1)$ and $p\in[1,+\infty)$, we let
\begin{equation*}
W^{\alpha,p}(\Omega;\R^m)
=\set*{u\in L^p(\Omega;\R^m) : [u]_{W^{\alpha,p}(\Omega;\,\R^m)}=\left(\int_\Omega\int_\Omega\frac{|u(x)-u(y)|^p}{|x-y|^{n+p\alpha}}\,dx\,dy\right)^{\frac{1}{p}}\!<+\infty}
\end{equation*}
be the space of $m$-vector-valued fractional Sobolev functions on~$\Omega$, see~\cite{DiNPV12} for its precise definition and main properties. We also let 
\begin{equation*}
w^{\alpha,p}(\Omega; \R^m)=\set*{u\in L^p_{\loc}(\Omega;\R^m) : [u]_{W^{\alpha,p}(\Omega;\,\R^m)}<+\infty}.
\end{equation*}
For $\alpha\in(0,1)$ and $p=+\infty$, we simply let
\begin{equation*}
W^{\alpha,\infty}(\Omega;\R^m)=\set*{u\in L^\infty(\Omega;\R^m) : \sup_{x,y\in \Omega,\, x\neq y}\frac{|u(x)-u(y)|}{|x-y|^\alpha}<+\infty},
\end{equation*}
so that $W^{\alpha,\infty}(\Omega;\R^m)=C^{0,\alpha}_b(\Omega;\R^m)$, the space of $m$-vector-valued bounded $\alpha$-H\"older continuous functions on~$\Omega$.

In order to avoid heavy notation, if the elements of a function space $F(\Omega;\R^m)$ are real-valued (i.e., $m=1$), then we will drop the target space and simply write~$F(\Omega)$.

Given $\alpha\in(0,n)$, we let
\begin{equation}\label{eq:Riesz_potential_def} 
I_{\alpha} f(x) 
= 
2^{-\alpha} \pi^{- \frac{n}{2}} \frac{\Gamma\left(\frac{n-\alpha}2\right)}{\Gamma\left(\frac\alpha2\right)}
\int_{\R^{n}} \frac{f(y)}{|x - y|^{n - \alpha}} \, dy, 
\quad
x\in\R^n,
\end{equation}
be the Riesz potential of order $\alpha$ of $f\in C^\infty_c(\R^n;\R^m)$. We recall that, if $\alpha,\beta\in(0,n)$ satisfy $\alpha+\beta<n$, then we have the following \emph{semigroup property}
\begin{equation}\label{eq:Riesz_potential_semigroup}
I_{\alpha}(I_\beta f)=I_{\alpha+\beta}f
\end{equation}
for all $f\in C^\infty_c(\R^n;\R^m)$. In addition, if $1<p<q<+\infty$ satisfy 
$
\frac{1}{q}=\frac{1}{p}-\frac{\alpha}{n},	
$
then there exists a constant $C_{n,\alpha,p}>0$ such that the operator in~\eqref{eq:Riesz_potential_def} satisfies
\begin{equation}\label{eq:Riesz_potential_boundedness}
\|I_\alpha f\|_{L^q(\R^n;\,\R^m)}\le C_{n,\alpha,p}\|f\|_{L^p(\R^n;\,\R^m)}
\end{equation}
for all $f\in C^\infty_c(\R^n;\,\R^m)$. As a consequence, the operator in~\eqref{eq:Riesz_potential_def} extends to a linear continuous operator from $L^p(\R^n;\R^m)$ to $L^q(\R^n;\R^m)$, for which we retain the same notation. For a proof of~\eqref{eq:Riesz_potential_semigroup} and~\eqref{eq:Riesz_potential_boundedness}, we refer the reader to~\cite{S70}*{Chapter~V, Section~1} and to~\cite{G14-M}*{Section~1.2.1}.

Given $\alpha\in(0,1)$, we also let
\begin{equation}\label{eq:def_frac_Laplacian}
(- \Delta)^{\frac{\alpha}{2}} f(x) 
= 
\nu_{n, \alpha}
\int_{\R^n} \frac{f(x + y) - f(x)}{|y|^{n + \alpha}}\,dy,
\quad
x\in\R^n,
\end{equation}
be the fractional Laplacian (of order~$\alpha$) of $f \in\Lip_b(\R^{n};\R^m)$, where
\begin{equation*}
\nu_{n,\alpha}=2^\alpha\pi^{-\frac n2}\frac{\Gamma\left(\frac{n+\alpha}{2}\right)}{\Gamma\left(-\frac{\alpha}{2}\right)},
\quad
\alpha\in(0,1).
\end{equation*}

Finally, we let
\begin{equation}\label{eq:def_Riesz_transform}
R f(x)
=
\pi^{-\frac{n+1}2}\,\Gamma\left(\tfrac{n+1}{2}\right)\,\lim_{\eps\to0^+}\int_{\set*{|y|>\eps}}\frac{y\,f(x+y)}{|y|^{n+1}}\,dy,
\quad
x\in\R^n,
\end{equation}
be the (vector-valued) \emph{Riesz transform} of a (sufficiently regular) function~$f$. 
We refer the reader to~\cite{G14-M}*{Sections~2.1 and~2.4.4}, \cite{S70}*{Chapter~III, Section~1} and~\cite{S93}*{Chapter~III} for a more detailed exposition. 
We warn the reader that the definition in~\eqref{eq:def_Riesz_transform} agrees with the one in~\cites{S93} and differs from the one in~\cites{G14-M,S70} for a minus sign, so that $R=\nabla I_1$ on $C^\infty_c(\R^n)$ in particular.
The Riesz transform~\eqref{eq:def_Riesz_transform} is a singular integral of convolution type, thus in particular it defines a continuous operator $R\colon L^p(\R^n)\to L^p(\R^n;\R^{n})$ for any given $p\in(1,+\infty)$, see~\cite{G14-C}*{Corollary~5.2.8}.
We also recall that its components $R_i$ satisfy
\begin{equation*}
\sum_{i=1}^nR_i^2=-\mathrm{Id}
\quad
\text{on}\ L^2(\R^n),
\end{equation*}
see~\cite{G14-C}*{Proposition~5.1.16}.

\subsection{The operators \texorpdfstring{$\nabla^\alpha$}{nablaˆalpha} and \texorpdfstring{$\div^\alpha$}{divˆalpha}} 

We briefly recall the definitions and the essential features of the non-local operators~$\nabla^\alpha$ and~$\diverg^\alpha$, see~\cites{S19,CS19,CS19-2,BCCS20} and~\cite{P16}*{Section~15.2}.

Let $\alpha\in(0,1)$ and set 
\begin{equation*}
\mu_{n, \alpha} 
= 
2^{\alpha}\, \pi^{- \frac{n}{2}}\, \frac{\Gamma\left ( \frac{n + \alpha + 1}{2} \right )}{\Gamma\left ( \frac{1 - \alpha}{2} \right )}.
\end{equation*}
We let
\begin{equation*}
\nabla^{\alpha} f(x) 
=
\mu_{n, \alpha} \lim_{\eps \to 0^+} \int_{\{ |y| > \eps \}} \frac{y \, f(x + y)}{|y|^{n + \alpha + 1}} \, dy,
\quad
x\in\R^n,
\end{equation*}
be the \emph{fractional $\alpha$-gradient} of $f\in\Lip_c(\R^n)$ and, similarly, we let
\begin{equation*}
\div^{\alpha} \varphi(x) 
= 
\mu_{n, \alpha} \lim_{\eps \to 0^+} \int_{\{ |y| > \eps \}} \frac{y \cdot \varphi(x + y)}{|y|^{n + \alpha + 1}} \, dy,
\quad
x\in\R^n,
\end{equation*}
be the \emph{fractional $\alpha$-divergence} of $\phi\in\Lip_c(\R^n;\R^n)$. 
The non-local operators~$\nabla^\alpha$ and~$\diverg^\alpha$ are well defined in the sense that the involved integrals converge and the limits exist. Moreover, since 
\begin{equation*}
\int_{\set*{|z| > \eps}} \frac{z}{|z|^{n + \alpha + 1}} \, dz=0,
\quad
\forall\eps>0,
\end{equation*}
it is immediate to check that $\nabla^{\alpha}c=0$ for all $c\in\R$ and
\begin{align*}
\nabla^{\alpha} f(x)
&=\mu_{n, \alpha} \int_{\R^{n}} \frac{(y - x)  (f(y) - f(x)) }{|y - x|^{n + \alpha + 1}} \, dy,
\quad
x\in\R^n,
\end{align*}
for all $f\in\Lip_c(\R^n)$. Analogously, we have
\begin{align*}
\div^{\alpha} \varphi(x) 
&= \mu_{n, \alpha} \int_{\R^{n}} \frac{(y - x) \cdot (\varphi(y) - \varphi(x)) }{|y - x|^{n + \alpha + 1}} \, dy,
\quad
x\in\R^n,
\end{align*}
for all $\phi\in\Lip_c(\R^n)$.
From the above expressions, it is not difficult to recognize that, given $f\in\Lip_c(\R^n)$ and $\phi\in\Lip_c(\R^n;\R^n)$, it holds that
\begin{equation*}
\nabla^\alpha f\in L^p(\R^n;\R^n)
\quad\text{and}\quad
\div^\alpha\phi\in L^p(\R^n)	
\end{equation*}
for all $p\in[1,+\infty]$, see~\cite{CS19}*{Corollary~2.3}.
Finally, the fractional operators $\nabla^\alpha$ and $\div^\alpha$ are \emph{dual}, in the sense that
\begin{equation*}
\int_{\R^n}f\,\div^\alpha\phi \,dx=-\int_{\R^n}\phi\cdot\nabla^\alpha f\,dx
\end{equation*}
for all $f\in\Lip_c(\R^n)$ and $\phi\in\Lip_c(\R^n;\R^n)$, see~\cite{Sil19}*{Section~6} and~\cite{CS19}*{Lemma~2.5}.

With a slight abuse of notation, in the following we let $\nabla^1$ and $\div^1$ be the usual (local) gradient and divergence.
Note that this notation is coherent with the asymptotic behavior of the fractional operators $\nabla^\alpha$ and $\div^\alpha$ when $\alpha\to1^-$ for sufficiently regular functions, see the analysis made in~\cite{CS19-2}.

\section{The \texorpdfstring{$BV^{\alpha,p}(\R^n)$}{BVˆ(alpha,p)(Rˆn)} space}

\label{sec:BV_alpha_p_space}

In this section we study the main properties of the $BV^{\alpha,p}$ functions, following the strategy adopted in~\cite{CS19}*{Section~3}.

\subsection{Definition of \texorpdfstring{$BV^{\alpha,p}(\R^n)$}{BVˆ(alpha,p)(Rˆn)}}

Let $\alpha\in(0,1]$ and $p\in[1,+\infty]$.
We say that a function $f\in L^p(\R^n)$ belongs to the space $BV^{\alpha,p}(\R^n)$ if
\begin{equation}
\label{eq:def_frac_var_p}
|D^\alpha f|(\R^n)
=
\sup\set*{\int_{\R^n}f\,\div^\alpha\phi\,dx : \phi\in C^\infty_c(\R^n;\R^n),\ \|\phi\|_{L^\infty(\R^n;\,\R^n)}\le1}
<+\infty,
\end{equation}
see~\cite{CS19}*{Section~3} for the case $p=1$ and the discussion in~\cite{CS19-2}*{Section~3.3} for the case $p\in(1,+\infty]$.	
In the case $p=1$, we simply write $BV^{\alpha,1}(\R^n)=BV^\alpha(\R^n)$.
The resulting linear space 
\begin{equation*}
BV^{\alpha,p}(\R^n)=
\set*{f\in L^p(\R^n) : |D^\alpha f|(\R^n)<+\infty}
\end{equation*}
endowed with the norm
\begin{equation*}
\|f\|_{BV^{\alpha,p}(\R^n)}
=
\|f\|_{L^p(\R^n)}+|D^\alpha f|(\R^n),
\quad
f\in BV^{\alpha,p}(\R^n),
\end{equation*}
is a Banach space and that the fractional variation defined in~\eqref{eq:def_frac_var_p} is lower semicontinuous with respect to the $L^p$-convergence. Similarly as it was proved in the case $p=1$ in \cite{CS19}*{Theorem 3.2}, it is possible to show the following result relating non-local distributional gradients of $BV^{\alpha,p}$ functions to vector valued Radon measures.

\begin{theorem}[Structure Theorem for $BV^{\alpha,p}$ functions]\label{th:structure_BV_alpha}
Let $\alpha\in(0,1), p \in [1, +\infty]$ and $f \in L^{p}(\R^{n})$. Then, $f \in BV^{\alpha,p}(\R^n)$ if and only if there exists a finite vector valued Radon measure $D^{\alpha} f \in \M(\R^n; \R^{n})$ such that 
\begin{equation}\label{eqi:BV_alpha_p_duality} 
\int_{\R^{n}} f\, \div^{\alpha} \varphi \, dx = - \int_{\R^n} \varphi \cdot d D^{\alpha} f 
\end{equation}
for all $\varphi \in C^{\infty}_{c}(\R^n; \R^{n})$. In addition, for any open set $U\subset\R^n$ it holds
\begin{equation*}
|D^{\alpha} f|(U) = \sup\set*{\int_{\R^n}f\,\div^\alpha\phi\ dx : \phi\in C^\infty_c(U;\R^n),\ \|\phi\|_{L^\infty(U;\R^n)}\le1}.
\end{equation*}
\end{theorem}

\subsection{Approximation by smooth functions}

Here and in the rest of the paper, we let $(\rho_\eps)\subset C^\infty_c(\R^n)$ be a family of standard mollifiers as in~\cite{CS19}*{Section~3.3}.
The following approximation theorem is the extension to $BV^{\alpha,p}$ functions of \cite{CS19}*{Lemma~3.5 and Theorem~3.7}.
We leave its proof to the interested reader.

\begin{theorem}[Approximation by $C^\infty\cap BV^{\alpha,p}$ functions] 
\label{res:moll_conv_alpha_p}
Let $\alpha \in (0,1]$ and  $p \in [1, +\infty]$.
Let $f\in BV^{\alpha,p}(\R^n)$ and define
$f_{\eps}= f*\rho_{\eps}$ for all $\eps>0$.
Then $(f_{\eps})_{\eps>0}\subset BV^{\alpha,p}(\R^n) \cap C^{\infty}(\R^n)$ with
$D^{\alpha} f_{\eps} = (\rho_{\eps} \ast D^{\alpha} f) \Leb{n}$
for all $\eps>0$.
Moreover, the following properties hold.
\begin{enumerate}[(i)]
\item 
If $p<+\infty$, then $f_{\eps} \to f$ in $L^p(\R^n)$ as $\eps\to0^+$;
if $p=+\infty$, then $f_{\eps} \to f$ in $L^q_{\loc}(\R^n)$ as $\eps\to0^+$ for all $q\in[1,+\infty)$;
\item 
$D^{\alpha} f_{\eps} \weakto D^{\alpha} f$ in $\M(\R^n; \R^{n})$
and 
$|D^{\alpha} f_{\eps}|(\R^n) \to |D^{\alpha} f|(\R^n)$
as $\eps\to0^+$. 
\end{enumerate}
\end{theorem}

The following result extends the approximation by test functions given in~\cite{CS19}*{Theorem~3.8} to functions in $BV^{\alpha,p}(\R^n)$ for $\alpha\in(0,1)$ and all exponents
$p\in\left[1,\frac n{n-\alpha}\right)$.
In the proof below and in the following, we let
\begin{equation}
\label{eq:def_D_alpha}
\mathcal D^\alpha f(x)
=
\int_{\R^n}
\frac{|f(x+h)-f(x)|}{|h|^{n+\alpha}}
\,dh,
\quad
x\in\R^n,
\end{equation}
for any $f\in\Lip_c(\R^n;\R^m)$, $m\in\N$.
Note that 
$|\nabla^\alpha f(x)|
\le
\mu_{n,\alpha}
\mathcal D^\alpha f(x)$
for all $x\in\R^n$ and that $\mathcal D^\alpha f\in L^p(\R^n)$ for all $p\in[1,+\infty]$.

\begin{theorem}[Approximation by $C^\infty_c$ functions]
\label{res:approx_by_test}
Let 
$\alpha\in(0,1)$ 
and 
$p\in\left[1,\frac n{n-\alpha}\right)$.
If $f\in BV^{\alpha,p}(\R^n)$, then there exists $(f_k)_{k\in\N}\subset C^\infty_c(\R^n)$ such that
\begin{enumerate}[(i)]

\item
$f_k\to f$ in $L^p(\R^n)$ as $k\to+\infty$;

\item
$|D^\alpha f_k|(\R^n)
\to
|D^\alpha f|(\R^n)$
as
$k\to+\infty$.

\end{enumerate}
\end{theorem}

\begin{proof}
Let $(\eta_R)_{R>0}\subset C^\infty_c(\R^n)$ be a family of cut-off functions such that 
\begin{equation*}
0\le\eta_R\le1,
\quad
\eta_R=1\ \text{on}\ B_R,
\quad
\supp(\eta_R)\subset\closure{B}_{2R},
\quad
\Lip(\eta_R)\le\frac 2R.
\end{equation*}
We can also assume that $\eta_R(x)=\eta_1(\frac xR)$ for all $x\in\R^n$ and $R>0$. 
The proof now goes as the one of~\cite{CS19}*{Theorem~3.8} with minor modifications.
We simply have to check that 
\begin{equation}
\label{eq:claim_scaling_test_approx}
\lim_{R\to+\infty}
\int_{\R^n} |f(x)|
\int_{\R^n}\frac{|\eta_R(y)-\eta_R(x)|}{|x-y|^{n+\alpha}}\,dy\,dx
=0.
\end{equation} 
Indeed, by H\"older's inequality, we have 
\begin{equation*}
\int_{\R^n} |f(x)|
\int_{\R^n}\frac{|\eta_R(y)-\eta_R(x)|}{|x-y|^{n+\alpha}}\,dy\,dx
\le
\|f\|_{L^p(\R^n)}
\|\mathcal D^\alpha\eta_R\|_{L^q(\R^n)},
\end{equation*}
where $\frac1p+\frac1q=1$, and a simple change of variables shows that
\begin{equation*}
\|\mathcal D^\alpha\eta_R(x)\|_{L^q(\R^n)}
=
R^{\frac nq-\alpha}\,\|\mathcal D^\alpha\eta_1\|_{L^q(\R^n)}
\end{equation*}
for all $R>0$.
The claim in~\eqref{eq:claim_scaling_test_approx} thus follows provided that $\frac nq-\alpha<0$, which is equivalent to 
$p\in\left[1,\frac n{n-\alpha}\right)$, and the proof is complete.
\end{proof}

For the sake of completeness, we also treat the case $\alpha=1$ of the previous result.

\begin{proposition}
\label{res:approx_by_test_BV}
Let $n\in\N$ and $p\in[1,+\infty)$ be such that $p\le\frac n{n-1}$ for $n\ge2$.
If $f\in BV^{1,p}(\R^n)$, then there exists $(f_k)_{k\in\N}\subset C^\infty_c(\R^n)$ such that
\begin{enumerate}[(i)]
\item
$f_k\to f$ in $L^p(\R^n)$ as $k\to+\infty$;
\item
$|D f_k|(\R^n)
\to
|D f|(\R^n)$
as
$k\to+\infty$.
\end{enumerate}
\end{proposition}

\begin{proof}
Thanks to \cref{res:moll_conv_alpha_p}, we can assume $f \in C^{\infty}(\R^n) \cap BV^{1,p}(\R^n)$ without loss of generality. 
Now let $(\eta_R)_{R>0}\subset C^\infty_c(\R^n)$ be a family of cut-off functions as in the proof of \cref{res:approx_by_test}. 
Clearly, $\eta_R f \to f$ in $L^p(\R^n)$ as $R \to + \infty$. 
Moreover, since
$\nabla (\eta_R f) = \eta_R \nabla f + f \nabla \eta_R$,
we thus just need to check that $\|f\, \nabla \eta_R\|_{L^1(\R^n;\,\R^n)} \to 0^+$ as $R \to + \infty$. 
Indeed, by H\"older's inequality, we can estimate
\begin{align*}
\|f\,\nabla\eta_R\|_{L^1(\R^n;\,\R^n)}
&=
\int_{B_{2R}\setminus B_R}|f|\,|\nabla\eta_R|
\,dx
\le
\frac2R	
\int_{B_{2R}\setminus B_R}|f|
\,dx
\\
&\le
\frac2R
\,
|B_{2R}\setminus B_R|^{1-\frac1p}
\|f\|_{L^p(B_{2R}\setminus B_R)}
\le
2\,
|B_{2}\setminus B_1|^{1-\frac1p}
\|f\|_{L^p(\R^n\setminus B_R)}
\,
R^{n-1-\frac np}
\end{align*}
and the conclusion immediately follows.
\end{proof}

\subsection{Gagliardo--Nirenberg--Sobolev inequality}

Thanks to the approximation by test functions given by \cref{res:approx_by_test}, we can extend~\cite{CS19}*{Theorem~3.9} and prove the analogue of the Gagliardo--Nirenberg--Sobolev inequality for the space $BV^{\alpha,p}(\R^n)$ whenever
$p\in\left[1,\frac n{n-\alpha}\right)$.

\begin{theorem}[Gagliardo--Nirenberg--Sobolev inequality]
\label{res:GNS_ineq}
Let $\alpha\in(0,1)$ and let
$p\in\left[1,\frac n{n-\alpha}\right)$.
There exists a constant $c_{n,\alpha}>0$, depending on $n$ and $\alpha$ only, such that
\begin{equation*}
\|f\|_{L^{\frac n{n-\alpha},r}(\R^n)}
\le
c_{n,\alpha}
|D^\alpha f|(\R^n)
\end{equation*}
for all $f\in BV^{\alpha,p}(\R^n)$, where $r=+\infty$ if $n=1$ and $r=1$ if $n\ge2$.
As a consequence, $BV^{\alpha,p}(\R^n)\subset L^q(\R^n)$ continuously for all $q\in\left[p,\frac n{n-\alpha}\right)$, with also $q=\frac n{n-\alpha}$ if $n\ge2$.
\end{theorem}

\begin{proof}
Assume that $f\in C^\infty_c(\R^n)$ to start. 
Arguing as in the proof of~\cite{CS19-2}*{Theorem~3.8}, we can estimate $
|f|
\le
c_{n,\alpha}\,
I_\alpha|\nabla^\alpha f|
$
for some constant $c_{n,\alpha}>0$ depending only on~$n$ and~$\alpha$ (possibly varying from line to line).
Thanks to the Hardy--Littlewood--Sobolev inequality, we immediately deduce that 
\begin{equation*}
\|f\|_{L^{\frac n{n-\alpha},\infty}(\R^n)}
\le
c_{n,\alpha}\,
\|\nabla^\alpha f\|_{L^1(\R^n; \R^n)}
\end{equation*}  
for all $f\in C^\infty_c(\R^n)$.
Moreover, if $n\ge2$, then we can apply~\cite{S20-An}*{Theorem~1.1} to the vector field $F=\nabla^\alpha f$ in order to get that 
\begin{equation*}
\|I_\alpha\nabla^\alpha f\|_{L^{\frac n{n-\alpha},1}(\R^n; \R^n)}
\le
c_{n,\alpha}\,
\|\nabla^\alpha f\|_{L^1(\R^n; \R^n)}
\end{equation*}
for all $f\in C^\infty_c(\R^n)$.
Since $I_\alpha\nabla^\alpha f = Rf$ for all $f\in C^\infty_c(\R^n)$ and $R\colon L^{\frac n{n-\alpha},1}(\R^n)\to L^{\frac n{n-\alpha},1}(\R^n; \R^n)$ strongly (recall the definition in~\eqref{eq:def_Riesz_transform} and the properties of the Riesz transform), we immediately deduce that \begin{equation*}
\|f\|_{L^{\frac n{n-\alpha},1}(\R^n)}
\le
c_{n,\alpha}\,
\|\nabla^\alpha f\|_{L^1(\R^n; \R^n)}
\end{equation*}  
for all $f\in C^\infty_c(\R^n)$, with $n\ge2$.
The conclusion then follows by combining a standard approximation argument exploiting \cref{res:approx_by_test} with Fatou's Lemma. 
\end{proof}

For $\alpha=1$, the previous result can be stated as follows.

\begin{proposition}[Alvino's inequality]
\label{res:alvino}
Let $n\in\N$ and $p\in[1,+\infty)$.
If $n\ge2$ and $p\le\frac n{n-1}$, then there exists a dimensional constant $c_n>0$ such that
\begin{equation*}
\|f\|_{L^{\frac n{n-1},1}(\R^n)}
\le
c_n\,|Df|(\R^n)
\end{equation*}
for all $f\in BV^{1,p}(\R^n)$.
If $n=1$, then
$\|f\|_{L^\infty(\R)}
\le
|Df|(\R)$ for all $f\in BV^{1,p}(\R)$.\end{proposition}

\begin{proof}
While the case $n=1$ is a well-known property of functions having bounded variation, the case $n\ge2$ follows from Alvino's inequality~\cite{A77} for functions in~$BV(\R^n)$ (also see~\cite{S20-New}*{Section~5}) in combination with \cref{res:approx_by_test_BV}.
We leave the details to the interested reader.
\end{proof}

\subsection{The space  \texorpdfstring{$S^{\alpha,p}(\R^n)$}{Sˆ(alpha,p)(Rˆn)} and the embedding \texorpdfstring{$BV^{\alpha,p}\subset S^{\beta,q}$}{of BVˆ(alpha,p) into Sˆ(beta,q)}}

Let $\alpha\in(0,1)$ and $p\in[1,+\infty]$.
We define the \emph{weak fractional $\alpha$-gradient} of a function $f\in L^p(\R^n)$ as the function $\nabla^\alpha f\in L^1_{\loc}(\R^n;\R^n)$ satisfying
\begin{equation*}
\int_{\R^n}f\,\div^\alpha\phi\, dx
=-\int_{\R^n}\nabla^\alpha f\cdot\phi\, dx
\end{equation*} 
for all $\phi\in C^\infty_c(\R^n;\R^n)$. 
We hence let the linear space 
\begin{equation*}S^{\alpha,p}(\R^n)
=
\set*{f\in L^p(\R^n) : \exists\, \nabla^\alpha f \in L^p(\R^n;\R^n)}
\end{equation*}
endowed with the norm
\begin{equation*}
\|f\|_{S^{\alpha,p}(\R^n)}
=
\|f\|_{L^p(\R^n)}+\|\nabla^\alpha f\|_{L^p(\R^n;\,\R^{n})},
\quad
f\in S^{\alpha,p}(\R^n),
\end{equation*}
be the \emph{distributional fractional Sobolev space}. 

As shown in~\cite{CS19}*{Proposition~3.20}, ($S^{\alpha,p}(\R^n)$,$\|\cdot\|_{S^{\alpha,p}(\R^n)}$) is a Banach space for all $p\in[1,+\infty]$.
Thanks to~\cite{CS19}*{Theorem~3.23} for $p=1$ and to~\cite{BCCS20}*{Theorem~A.1} for $p\in(1,+\infty)$ (we refer the reader also to~\cite{KS21}*{Theorem~2.7} for a simpler proof), the set $C^\infty_c(\R^n)$ is dense in $S^{\alpha,p}(\R^n)$. 
As a consequence, for $p\in(1,+\infty)$ it is possible to identify $S^{\alpha,p}(\R^n)$ with \emph{the fractional Bessel potential space $L^{\alpha,p}(\R^n)$}, see~\cite{BCCS20}*{Corollary~2.1} and the discussion therein.

We now want to provide a rigorous formulation of the na\"ive intuition that
\begin{equation*}
\text{\textit{if the order of differentiability decreases, then the order of integrability increases},}	
\end{equation*}
 that is to say, if $\nabla^\alpha f\in L^p(\R^n;\R^n)$ for some $\alpha\in(0,1)$ and $p\in[1,+\infty)$, then $\nabla^\beta f\in L^q(\R^n;\R^n)$ for some \emph{lower} fractional differentiation order $\beta<\alpha$ and some \emph{higher} integrability exponent~\mbox{$q>p$} (depending on $\beta$).

For $p>1$, the above principle is a simple consequence of the known embedding theorems between the fractional Bessel potential spaces, thanks to the aforementioned identification $S^{\alpha,p}(\R^n)=L^{\alpha,p}(\R^n)$.

The more delicate case $p=1$ is covered in
\cref{res:BV_alpha_p_in_S_beta_q} below. 
We refer the reader also to~\cite{CS19}*{Theorem~3.32} and to~\cite{CS19-2}*{Propositions~3.2(i), 3.3 and~3.12} for similar results in this direction.

\begin{theorem}[$BV^{\alpha,p}\subset S^{\beta,q}$ for $p<\frac n{n-\alpha}$]
\label{res:BV_alpha_p_in_S_beta_q}
Let $\alpha,\beta\in(0,1]$, with $\beta<\alpha$, and let $p, q \in[1,+\infty]$ be such that  $p \le q < \frac{n}{n+\beta-\alpha}$.
Then $BV^{\alpha,p}(\R^n)\subset S^{\beta,q}(\R^n)$ continuously.
\end{theorem}

\begin{proof}
Assume that $f\in C^\infty_c(\R^n)$ and let $R>0$.
Arguing as in the proof of~\cite{CS19-2}*{Proposition~3.12}, we can estimate
\begin{align*}
|\nabla^\beta f|(x)
\le
\frac{\mu_{n,1+\beta-\alpha}}{n+\beta-\alpha}\,
\bigg(
\int_{|h|<R}
\frac{|\nabla^\alpha f|(x+h)}{|h|^{n+\beta-\alpha}}\,dh
+
\bigg|
\int_{|h|\ge R}
\frac{\nabla^\alpha f(x+h)}{|h|^{n+\beta-\alpha}}\,dh\,
\bigg|\,
\bigg)
\end{align*}
for all $x\in\R^n$. 
On the one side, we can write 
\begin{align*}
\int_{|h|<R}
\frac{|\nabla^\alpha f|(x+h)}{|h|^{n+\beta-\alpha}}\,dh
=
\left(
\frac{\chi_{B_R}}{|\cdot|^{n+\beta-\alpha}}
*
|\nabla^\alpha f|\right)(x)
\end{align*}
for all $x\in\R^n$, so that
\begin{align*}
\bigg\|
\int_{|h|<R}
\frac{|\nabla^\alpha f|(x+h)}{|h|^{n+\beta-\alpha}}\,dh
\,
\bigg\|_{L^q(\R^n)}
&=
\bigg\|
\frac{\chi_{B_R}}{|\cdot|^{n+\beta-\alpha}}
*
|\nabla^\alpha f|\,
\bigg\|_{L^q(\R^n)}
\\
&\le
\left (\frac{n\omega_n}{n - (n + \beta - \alpha)q} \right )^{1/q}
\,
R^{\frac nq - (n+\beta-\alpha)}\,
\|\nabla^\alpha f\|_{L^1(\R^n;\,\R^n)}
\end{align*}
by Young's inequality. 
On the other side, arguing as in the proof of~\cite{CS19-2}*{Proposition~3.12}, we can write
\begin{align*}
\int_{|h|\ge R}
\frac{\nabla^\alpha f(x+h)}{|h|^{n+\beta-\alpha}}\,dh
=
\int_{\R^n}
\bigg(
\frac{f(x+Ry)}{R^\beta}
-
(n+\beta-\alpha)
\int_R^{+\infty}
\frac{f(x+ry)}{r^{\beta+1}}
\,dr
\bigg)\,
d D^\alpha\chi_{B_1}(y)
\end{align*}
for all $x\in\R^n$, so that 
\begin{align*}
\bigg\|
\int_{|h|\ge R}
\frac{\nabla^\alpha f(\cdot+h)}{|h|^{n+\beta-\alpha}}\,dh\,
\bigg\|_{L^q(\R^n;\,\R^n)}
\le
c_{n,\alpha,\beta}\,
R^{-\beta}\,
\|f\|_{L^q(\R^n)}
\end{align*}
by Minkowski's integral inequality, for some constant $c_{n,\alpha,\beta}>0$ depending only on~$n$, $\alpha$ and~$\beta$.
Hence we get that 
\begin{equation*}
\|\nabla^\beta f\|_{L^q(\R^n;\,\R^n)}
\le
c_{n,\alpha,\beta,q}
\left(
R^{\frac nq - (n+\beta-\alpha)}\,
\|\nabla^\alpha f\|_{L^1(\R^n;\,\R^n)}
+
R^{-\beta}\,
\|f\|_{L^q(\R^n)}
\right)
\end{equation*}
whenever $R>0$, for some constant $c_{n,\alpha,\beta,q}>0$ depending only on~$n$, $\alpha$, $\beta$ and~$q$. 
Choosing 
$R
=
\left(
\frac{\|f\|_{L^q(\R^n)}}{\|\nabla^\alpha f\|_{L^1(\R^n; \R^n)}}\right)^{\frac{1}{\frac nq-n+\alpha}}$, we get that 
\begin{equation*}
\|\nabla^\beta f\|_{L^q(\R^n;\,\R^n)}
\le
c_{n,\alpha,\beta,q}
\,
\|f\|_{L^q(\R^n)}^{1-\frac{\beta q}{n-q(n-\alpha)}}\,
\|\nabla^\alpha f\|_{L^1(\R^n;\,\R^n)}^{\frac{\beta q}{n-q(n-\alpha)}}
\end{equation*}
for all $f\in C^\infty_c(\R^n)$. 
The conclusion thus follows from \cref{res:approx_by_test} and \cref{res:GNS_ineq} (\cref{res:approx_by_test_BV} and  \cref{res:alvino} in the case $\alpha =1$) via a routine approximation argument, since clearly $p < \frac{n}{n - \alpha}$. 
\end{proof}

\subsection{Generalized integration-by-parts formula for \texorpdfstring{$BV^{\alpha,p}$}{BVˆ(alpha,p)} functions}

The following result is a generalization of the fractional integration-by-parts formula \eqref{eqi:BV_alpha_p_duality} (the case $p=1$ was actually already analyzed in~\cite{CS19-2}*{Proposition~2.7}).
This result will be useful for integrating by parts  $BV^{\alpha,p}$ functions on balls, see \cref{res:int_by_parts_E_ball} below.

\begin{proposition}[$W^{1,q}\cap C_b$-regular test] \label{res:Lip_test_BV_alpha_p}
Let $\alpha \in (0, 1)$ and let $p,q \in [1,+\infty]$ be such that $\frac1p+\frac1q=1$.
If $f \in BV^{\alpha, p}(\R^n)$, then
\begin{equation}
\label{eq:BV_alpha_p_duality}
\int_{\R^{n}} f\, \div^{\alpha} \phi \, dx 
= 
- \int_{\R^{n}} \phi \cdot d D^{\alpha} f
\end{equation}
for all $\phi\in W^{1,q}(\R^n;\R^n)\cap C_b(\R^n;\R^n)$ if $q > 1$, and for all $\phi\in BV(\R^n;\R^n)\cap C_b(\R^n;\R^n)$ if $q=1$.
\end{proposition}

\begin{proof} 
We start by noticing that, under the above assumptions, $\div^\alpha\phi \in L^{q}(\R^n)$, thanks to \cite{CS19-2}*{Propositions 3.2 and 3.3}. Then, we divide the proof into two steps and adopt the same strategy of~\cite{CS19}*{Theorem~3.8} and~\cite{CS19-2}*{Proposition~2.7} with minor modifications.

\smallskip

\textit{Step~1}. 
Assume 
$\phi\in W^{1,q}(\R^n;\R^n)\cap\Lip_b(\R^n;\R^n)\cap C^\infty(\R^n; \R^n)$
and let $(\eta_R)_{R>0}\subset C^\infty_c(\R^n)$ be a family of cut-off functions as in~\cite{CS19}*{Section~3.3}.
On the one hand, since
\begin{equation*}
\abs*{\int_{\R^n}f\eta_R\,\div^\alpha\phi\,dx-\int_{\R^n}f\,\div^\alpha\phi\,dx}
\le
\|\div^\alpha\phi\|_{L^q(\R^n)}
\|f(1-\eta_R)\|_{L^p(\R^n)}
\end{equation*}
for all $R>0$, by Lebesgue's Dominated Convergence Theorem we have
\begin{equation*}
\lim_{R\to+\infty}\int_{\R^n}f\eta_R\,\div^\alpha\phi\,dx
=\int_{\R^n}f\,\div^\alpha\phi\,dx.
\end{equation*}
On the other hand, by~\cite{CS19-2}*{Lemmas~2.2 and~2.5} we can write
\begin{align*}
\int_{\R^n}f\eta_R\,\div^\alpha\phi\,dx
=\int_{\R^n}f\,\div^\alpha(\eta_R\phi)\,dx
-\int_{\R^n}f\,\phi\cdot\nabla^\alpha\eta_R\,dx
-\int_{\R^n}f\,\div^\alpha_{\rm NL}(\eta_R,\phi)\,dx
\end{align*} 
for all $R>0$.
Since $\phi\eta_R\in C^\infty_c(\R^n; \R^n)$,  \eqref{eq:BV_alpha_p_duality} implies that
\begin{equation*}
\int_{\R^n}f\,\div^\alpha(\eta_R\phi)\,dx
=-\int_{\R^n}\eta_R\phi\cdot dD^\alpha f
\end{equation*}
for all $R>0$.
Since
\begin{equation*}
\abs*{
\int_{\R^n}\eta_R\phi\cdot dD^\alpha f
-
\int_{\R^n}\phi\cdot dD^\alpha f
}
\le
\|\phi\|_{L^\infty(\R^n;\R^n)}\int_{\R^n}(1-\eta_R)\,d|D^\alpha f|
\end{equation*}
for all $R>0$, by Lebesgue's Dominated Convergence Theorem (with respect to the finite measure~$|D^\alpha f|$) we have
\begin{equation*}
\lim_{R\to+\infty}\int_{\R^n}\eta_R\,\phi\cdot dD^\alpha f
=\int_{\R^n}\phi\cdot dD^\alpha f.
\end{equation*}
Finally, on the one side we can estimate
\begin{align*}
\abs*{
\int_{\R^n}f\,\phi\cdot\nabla^\alpha\eta_R\,dx
}
\le
\mu_{n,\alpha}
\int_{\R^n}|f(x)|\,|\phi(x)|\int_{\R^n}\frac{|\eta_R(y)-\eta_R(x)|}{|y-x|^{n+\alpha}}\,dy\,dx
\end{align*}
for all $R>0$, while, on the other side,
\begin{equation*}
\abs*{
\int_{\R^n}f\,\div^\alpha_{\rm NL}(\eta_R,\phi)\,dx
}
\le
\mu_{n,\alpha}
\int_{\R^n}|f(x)|
\int_{\R^n}\left |\eta_R(y)-\eta_R(x) \right| \, \frac{|\phi(y)-\phi(x)|}{|y-x|^{n+\alpha}}\,dy\,dx
\end{equation*}
for all $R>0$. 
We claim that 
\begin{equation}
\label{eq:claim_DCT_1}
\lim_{R\to+\infty}
\int_{\R^n}|f(x)|\,|\phi(x)|\int_{\R^n}\frac{|\eta_R(y)-\eta_R(x)|}{|y-x|^{n+\alpha}}\,dy\,dx
=0.
\end{equation}
Indeed, $f\phi\in L^1(\R^n;\R^n)$ and~\eqref{eq:claim_DCT_1} follows by Lebesgue's Dominated Convergence Theorem.
We also claim that
\begin{equation}
\label{eq:claim_DCT_2}
\lim_{R\to+\infty}
\int_{\R^n}|f(x)|
\int_{\R^n}\left |\eta_R(y)-\eta_R(x) \right| \,\frac{|\phi(y)-\phi(x)|}{|y-x|^{n+\alpha}}\,dy
\,dx
=0.
\end{equation}
Indeed, since $\phi\in\Lip_b(\R^n;\R^n)$ and $\|\eta_R\|_{L^\infty(\R^n)}\le1$ for all $R>0$, by Lebesgue's Dominated Convergence Theorem we get that
\begin{equation}
\label{eq:scacco}
\lim_{R\to+\infty}
\int_{\R^n}\left |\eta_R(y)-\eta_R(x) \right| \,\frac{|\phi(y)-\phi(x)|}{|y-x|^{n+\alpha}}\,dy
=0
\end{equation}
for all $x\in\R^n$. 
Moreover, for a.e.\ $x\in\R^n$ we have
\begin{equation}
\label{eq:matto}
\int_{\R^n}\left |\eta_R(y)-\eta_R(x) \right| \,\frac{|\phi(y)-\phi(x)|}{|y-x|^{n+\alpha}}\,dy
\le
2\,\mathcal D^\alpha\phi(x).
\end{equation}
Therefore, combining~\eqref{eq:scacco} and~\eqref{eq:matto}, again by Lebesgue's Dominated Convergence Theorem we get~\eqref{eq:claim_DCT_2}.
Thus~\eqref{eq:BV_alpha_p_duality} is proved whenever 
$\phi\in W^{1,q}(\R^n;\R^n)\cap\Lip_b(\R^n;\R^n)\cap C^\infty(\R^n; \R^n)$.

\smallskip

\textit{Step~2}.
Now assume 
$\phi\in W^{1,q}(\R^n;\R^n)\cap C_b(\R^n;\R^n)$
(if $q=1$, then we instead take
$\phi\in BV(\R^n;\R^n)\cap C_b(\R^n;\R^n)$)
and let $(\rho_\eps)_{\eps>0}\subset C^\infty_c(\R^n)$ be a family of standard mollifiers as in~\cite{CS19}*{Section~3.3}. 
Then 
$\phi_\eps=\rho_\eps*\phi\in W^{1,q}(\R^n;\R^n)\cap\Lip_b(\R^n;\R^n)\cap C^\infty(\R^n; \R^n)$
and so, by Step~1, we can write
\begin{equation*}
\int_{\R^{n}} f\, \div^{\alpha} \phi_\eps \, dx 
= 
- \int_{\R^{n}} \phi_\eps \cdot d D^{\alpha} f
\end{equation*}
for all $\eps>0$.
On the one hand, it is not difficult to see that 
$\div^\alpha\phi_\eps=\rho_\eps*\div^\alpha\phi$ for all~$\eps>0$, so that
\begin{equation*}
\lim_{\eps\to0^+}
\int_{\R^{n}} f\, \div^{\alpha} \phi_\eps \, dx 
=
\int_{\R^{n}} f\, \div^{\alpha} \phi \, dx. 
\end{equation*} 
On the other hand, since $\phi\in C_b(\R^n;\R^n)$, we get
\begin{equation*}
\lim_{\eps\to0^+}
\int_{\R^{n}} |\phi_\eps-\phi| \, d |D^{\alpha} f|
=0
\end{equation*} 
by Lebesgue's Dominated Convergence Theorem (with respect to the finite measure $|D^\alpha f|$). 
This concludes the proof of~\eqref{eq:BV_alpha_p_duality}.
\end{proof}

\subsection{Integrability of \texorpdfstring{$\mathcal D^\alpha$}{Dˆalpha} in Lorentz space for \texorpdfstring{$BV^{1,p}$}{BVˆ(1,p)} functions}

We now need to focus on the weak integrability properties of the operator $\mathcal D^\alpha$ defined in~\eqref{eq:def_D_alpha} when applied to functions belonging to $BV^{1,p}(\R^n)$. 
This analysis will be useful for studying the integrability properties of the fractional gradient of the characteristic function of a ball, see \cref{res:fract_grad_estimate} below.
This result, in turn, will be useful in the proof of the integration-by-parts formula of $BV^{\alpha,p}$ functions on balls in \cref{res:int_by_parts_E_ball}.

\begin{theorem}[Weak integrability of $\mathcal D^\alpha$] 
\label{res:optimal_estimate}
Let $\alpha\in(0,1)$, $p\in[1,+\infty]$ and define
\begin{equation*}
p_\alpha=
\begin{cases}
\frac{p}{1-\alpha+\alpha p} 
& 
\text{if}\ p\in[1,+\infty),
\\[3mm]
\frac{1}{\alpha} 
& 
\text{if}\ p=+\infty.
\end{cases}
\end{equation*}
The operator $\mathcal D^\alpha\colon BV^{1,p}(\R^n)\to L^{p_\alpha,\infty}(\R^n)$ is well defined and satisfies
\begin{equation}
\label{eq:Spector_BV}
\|\mathcal D^\alpha f\|_{L^{p_\alpha,\infty}(\R^n)}
\le c_{n,\alpha,p}\,
\|f\|_{L^p(\R^n)}^{1-\alpha}
|Df|(\R^n)^\alpha
\end{equation}
for all $f\in BV^{1,p}(\R^n)$, where $c_{n,\alpha,p}>0$ is a constant depending on $n$, $\alpha$ and $p$ only.
\end{theorem}

\begin{proof}
The case $p=1$ is easy, since~\eqref{eq:Spector_BV} holds in the following stronger form \begin{equation*}
\|\mathcal D^\alpha f\|_{L^1(\R^n)}
\le 
c_{n,\alpha}\,
\|f\|_{L^1(\R^n)}^{1-\alpha}
|Df|(\R^n)^\alpha
\end{equation*}
for all $f\in BV(\R^n)$, whose simple proof is left to the reader (for instance, one can follow the strategy of the proof of~\cite{DiNPV12}*{Proposition~2.2}). 
In the following, we thus assume that $p>1$. 
We now divide the proof into three steps.

\smallskip

\textit{Step~1}. 
Assume 
$f\in W^{1,1}(\R^n)\cap C^\infty(\R^n)\cap\Lip_b(\R^n)$. 
By~\cite{S20-An}*{Lemma~3.2}, there exists a constant $C_{n, \alpha} > 0$ such that
\begin{equation*}
\mathcal D^\alpha f(x)
\le C_{n, \alpha}
(Mf(x))^{1-\alpha}\,
(M|\nabla f|(x))^\alpha
\end{equation*}
for all $x\in\R^n$.
Since $f\in L^p(\R^n)$, by H\"older's inequality in Lorentz spaces (see~\cite{O63}*{Theorem~3.4} and~\cite{S20-New}*{Theorem~5.1}) and the well-known continuity properties of the maximal function, we get that
\begin{align*}
\|\mathcal D^\alpha f\|_{L^{p_\alpha,\infty}(\R^n)}
&\le 
C_{n, \alpha}\,
\|(Mf)^{1-\alpha}\,
(M|\nabla f|)^\alpha
\|_{L^{p_\alpha,\infty}(\R^n)}
\\
&\le
\tfrac{C_{n,\alpha}\,p_\alpha}{p_\alpha-1}\,
\|(Mf)^{1-\alpha}\|_{L^{\frac{p}{1-\alpha}}(\R^n)}\,
\|(M|\nabla f|)^\alpha\|_{L^{\frac1\alpha,\infty}(\R^n)}
\\
&\le 
c_{n, \alpha,p}\,
\|f\|_{L^p(\R^n)}^{1-\alpha}\,
|Df|(\R^n)^\alpha,
\end{align*}
where 
$c_{n,\alpha,p}>0$ is a constant depending only on $n$, $\alpha$ and $p$.

\smallskip

\textit{Step~2}. 
Now assume $f\in BV(\R^n)\cap L^p(\R^n)$ and define $f_\eps=f*\rho_\eps$ for all $\eps>0$. 
Then 
$f_\eps\in W^{1,1}(\R^n)\cap C^\infty(\R^n)\cap\Lip_b(\R^n)$ 
with 
$\|f_\eps\|_{L^p(\R^n)}\le\|f\|_{L^p(\R^n)}$ 
for all $\eps>0$ 
and $|Df_\eps|(\R^n)\to|Df|(\R^n)$ 
as $\eps\to0^+$. 
Moreover, by the Fatou Lemma, we have 
\begin{equation*}
\mathcal D^\alpha f(x)
\le
\liminf_{\eps\to0^+}
\mathcal D^\alpha f_\eps(x)
\end{equation*}
for a.e.\ $x\in\R^n$. 
Hence, by~\cite{G14-C}*{Exercise~1.1.1(b)}, we get
\begin{align*}
\|\mathcal D^\alpha f\|_{L^{p_\alpha,\infty}(\R^n)}
&\le
\liminf_{\eps\to0^+}
\|\mathcal D^\alpha f_\eps\|_{L^{p_\alpha,\infty}(\R^n)}\\
&\le 
c_{n, \alpha,p}
\lim_{\eps\to0^+}
\|f_\eps\|_{L^p(\R^n)}^{1-\alpha}\,
|Df_\eps|(\R^n)^\alpha\\
&\le 
c_{n, \alpha,p}\,
\|f\|_{L^p(\R^n)}^{1-\alpha}\,
|Df|(\R^n)^\alpha
\end{align*} 
thanks to Step~1.

\smallskip

\textit{Step~3}. 
Finally, assume $f\in BV^{1,p}(\R^n)$.
Let $(\eta_R)_{R>0}\subset C^\infty_c(\R^n)$ be a family of cut-off functions as in~\cite{CS19}*{Section~3.3} and define $f_R=f\eta_R$ for all $R>0$.
Then $f_R\in BV(\R^n)\cap L^p(\R^n)$ with 
$\|f_R\|_{L^p(\R^n)}\le\|f\|_{L^p(\R^n)}$ 
for all $R>0$ 
and $|Df_R|(\R^n)\to|Df|(\R^n)$ 
as $R\to+\infty$.
Moreover, by the Fatou Lemma, we have 
\begin{equation*}
\mathcal D^\alpha f(x)
\le
\liminf_{R\to+\infty}
\mathcal D^\alpha f_R(x)
\end{equation*}
for a.e.\ $x\in\R^n$. 
Inequality~\eqref{eq:Spector_BV} thus follows again by~\cite{G14-C}*{Exercise~1.1.1(b)}, thanks to Step~2.
This concludes the proof.
\end{proof}

From \cref{res:optimal_estimate}, we immediately deduce the following integrability properties of the fractional gradient of the indicator function of a ball. 

\begin{corollary}[Integrability of $\nabla^\alpha\chi_{B_r(x)}$]
\label{res:fract_grad_estimate}
Let $\alpha\in(0,1)$ and let 
$p \in \left [1, \frac{1}{\alpha} \right )$.
There exists a constant $c_{n,\alpha,p} > 0$, depending on $n$, $\alpha$ and $p$ only, such that
\begin{equation} \label{eq:ball_fract_grad_estimate}
\|\nabla^{\alpha} \chi_{B_{r}(x)}\|_{L^{p}(\R^{n};\,\R^{n})} 
= c_{n, \alpha, q}\, r^{\frac{n}{p} - \alpha}
\end{equation}
for all $x \in \R^{n}$ and $r > 0$.
\end{corollary}

\begin{proof}
Let $x \in \R^{n}$ and $r > 0$ be fixed.
Since 
\begin{equation*}
\nabla^{\alpha} \chi_{B_r(x)}(y) = r^{-\alpha} (\nabla^{\alpha} \chi_{B_1(0)})\left (\frac{y - x}{r}\right)
\end{equation*}
by the rescaling property of $\nabla^\alpha$, we immediately get that
\begin{equation*}
\|\nabla^{\alpha} \chi_{B_r(x)}\|_{L^p(\R^n; \,\R^n)} 
=
\|\nabla^{\alpha} \chi_{B_1}\|_{L^p(\R^n;\, \R^n)}\, r^{\frac{n}{p} - \alpha}.
\end{equation*}
Since 
$
\nabla^{\alpha} \chi_{B_1} 
\in 
L^{1}(\R^n; \R^n)
\cap
L^{\frac{1}{\alpha}, \infty}(\R^n;\R^n)
$
by~\cite{CS19}*{Proposition~4.8} and \cref{res:optimal_estimate}, the conclusion follows by observing that $L^1(\R^n)\cap L^{\frac1\alpha,\infty}(\R^n)\subset L^p(\R^n)$ with continuous inclusion for all $p\in\left[1,\frac1\alpha\right)$.
This interpolation result can be proved for instance by arguing as in the proof of~\cite{G14-C}*{Proposition~1.1.14} with some minor modifications. We leave the simple details to the interested reader.
\end{proof}

\begin{remark}[The case $n=1$ in \cref{res:fract_grad_estimate}]
The estimate in~\eqref{eq:ball_fract_grad_estimate} in the case $n = 1$ can be obtained by a direct computation from the explicit formula
\begin{equation*}
\nabla^{\alpha} \chi_{(x - r,\, x + r)}(y) 
= 
\frac{\mu_{1, \alpha}}{\alpha} \left ( |y - x + r|^{-\alpha} - |y - x - r|^{-\alpha} \right ),
\quad
y\in\R,
\end{equation*}
given by \cite{CS19}*{Example 4.11}.
In particular, we deduce that
\begin{equation*}
\nabla^{\alpha} \chi_{(x -r,\, x + r)} \in L^{\frac{1}{\alpha}, \infty}(\R;\R) \setminus L^{\frac{1}{\alpha}, s}(\R;\R)
\end{equation*} 
for all $s \in [1,+\infty)$.
\end{remark}

From \cref{res:alvino}, we immediately deduce the following improvement of \cref{res:optimal_estimate}.
We leave its simple proof to the reader. 

\begin{corollary}[Improved weak integrability of $\mathcal D^\alpha$]
\label{res:improved_optimal_estimate}
Let $\alpha\in(0,1)$, $n\in\N$ and $p\in[1,+\infty)$ be such that $p\le\frac n{n-1}$ for $n\ge2$.
If $f\in BV^{1,p}(\R^n)$, then 
\begin{equation*}\mathcal D^\alpha f\in 
L^{\frac{p}{1-\alpha+\alpha p} 
,\infty}(\R^n)
\cap
L^{\frac n{n+\alpha-1},\infty}(\R^n).
\end{equation*}
\end{corollary}

\subsection{Integration by parts of \texorpdfstring{$BV^{\alpha,p}$}{BVˆ(alpha,p)} functions on balls}

We are now ready to state and prove the following in\-te\-gra\-tion-by-parts of $BV^{\alpha,p}$ functions on balls, which is a generalization of~\cite{CS19}*{Theorem~5.2} to $BV^{\alpha,p}$ functions for
$p\in\left(\frac{1}{1-\alpha},+\infty\right]$.
This result will be the central ingredient of the proof of the decay estimates for $BV^{\alpha,p}$ functions in \cref{res:decay_estimate_p} below.

\begin{theorem}[Integration by parts on balls] \label{res:int_by_parts_E_ball} 
Let $\alpha \in (0, 1)$ and 
$p\in\left(\frac{1}{1-\alpha},+\infty\right]$. 
If $f \in BV^{\alpha, p}(\R^{n})$, 
$\phi \in\Lip_c(\R^{n}; \R^{n})$
and $x \in \R^n$, then 
\begin{equation} 
\label{eq:int_by_parts_u_ball} 
\int_{B_r(x)} f \, \div^{\alpha} \phi \, dy 
+ 
\int_{\R^{n}} f \, \phi\, \cdot \nabla^{\alpha} \chi_{B_r(x)}\, dy  
+ 
\int_{\R^{n}} f \, \div^{\alpha}_{\rm NL} (\chi_{B_r(x)}, \phi) \, dy
=
-\int_{B_r(x)} \phi \cdot \, d D^{\alpha} f
\end{equation}
for $\Leb{1}$-a.e.\ $r > 0$.
\end{theorem}

\begin{proof}%[Proof of \cref{res:int_by_parts_E_ball}] 
Let $x \in \R^{n}$ and $\phi \in \Lip_{c}(\R^{n}; \R^{n})$ be fixed.
We divide the proof into two parts, dealing with the cases $p = +\infty$ and 
$p\in\left(\frac{1}{1-\alpha},+\infty\right)$ separately.

\smallskip

\textit{Case~1: $p = +\infty$}.
Let $\eps>0$ and define the function $h_{\eps, r, x}\in\Lip(\R^n)$ by setting
\begin{equation*} 
h_{\eps, r, x}(y) 
= 
\begin{cases} 
1 
& \text{if} \ 0 \le |y - x| \le r, \\[3mm]
\dfrac{r + \eps - |y - x|}{\eps} 
& \text{if} \ r < |y - x| < r + \eps, \\[4mm] 
0 
& \text{if} \  |y -x| \ge r + \eps, 
\end{cases} 
\end{equation*}
for all $y\in\R^n$.
By~\cite{CS19}*{Lemma 5.1}, we know that  
$\nabla^\alpha h_{\eps,r,x}\in L^1(\R^n;\R^n)$ with
\begin{equation} \label{eq:frac_grad_cutoff_ball} 
\nabla^{\alpha} h_{\eps, r, x}(y) 
= 
\frac{\mu_{n, \alpha}}{\eps (n + \alpha - 1)} 
\int_{B_{r + \eps}(x) \setminus B_r(x)} 
\frac{x - z}{|x - z|}\,|z - y|^{1 - n - \alpha}\, dz 
\end{equation}
for $\Leb{n}$-a.e.\ $y \in \R^{n}$.
On the one hand, since $h_{\eps, r, x}\,\phi\in\Lip_c(\R^n;\R^n)$, by \cref{res:Lip_test_BV_alpha_p} we have 
\begin{equation} \label{eq:IBP_E_1} 
\int_{\R^{n}} f \, \div^{\alpha}(h_{\eps, r, x}\,\phi) \, dy 
= 
- \int_{\R^{n}} h_{\eps, r, x}\,\phi \cdot \, d D^{\alpha} f.
\end{equation}
Since $h_{\eps, r, x}(y) \to \chi_{\closure{B_r(x)}}(y)$ as $\eps \to 0^+$ for all $y \in \R^{n}$ and $|D^{\alpha} f|(\de B_r(x)) = 0$ for $\Leb{1}$-a.e.\ $r > 0$, we can compute 
\begin{equation*}
\lim_{\eps\to0^+}\int_{\R^n} (h_{\eps, r, x}\,\phi) \cdot \, d D^{\alpha} f
=\int_{B_r(x)} \phi \cdot \, d D^{\alpha} f
\end{equation*}
for $\Leb{1}$-a.e.\ $r > 0$.
On the other hand, by~\cite{CS19-2}*{Lemma~2.5}, we have
\begin{equation} \label{eq:Leibniz_rule_h_phi} 
\div^{\alpha}(h_{\eps, r, x}\,\phi) 
= 
h_{\eps, r, x}\, \div^{\alpha}\phi 
+ 
\phi \cdot \nabla^{\alpha} h_{\eps, r, x} 
+ 
\div_{\rm NL}^{\alpha}(h_{\eps, r, x}, \phi). 
\end{equation}
We deal with each term of the right-hand side of~\eqref{eq:Leibniz_rule_h_phi} separately. For the first term, since $0\le h_{\eps, r, x}\le\chi_{B_{r+1}(x)}$ for all $\eps\in(0,1)$ and $h_{\eps, r, x} \to \chi_{B_r(x)}$ in $L^{q}(\R^{n})$ as $\eps \to 0^+$ for any $q \in [1,+\infty)$, by \cite{CS19}*{Corollary 2.3} and Lebesgue's Dominated Convergence Theorem we can compute
\begin{equation} \label{eq:IBP_conv_1} 
\lim_{\eps\to0^+}
\int_{\R^{n}} f \, h_{\eps, r, x} \,\div^{\alpha}\phi \, dy 
=
\int_{B_r(x)} f \, \div^{\alpha}\phi \, dy. 
\end{equation}
For the second term, by~\eqref{eq:frac_grad_cutoff_ball} we have
\begin{align*} 
\int_{\R^{n}} f(y) \, \phi(y) 
&\cdot 
\nabla^{\alpha} h_{\eps, r, x}(y) \, dy
\\ 
&= 
\frac{\mu_{n, \alpha}}{\eps (n + \alpha - 1)} 
\int_{\R^{n}} f(y) \, \phi(y) 
\cdot 
\int_{B_{r+\eps}(x) \setminus B_r(x)} \frac{x - z}{|x - z|} |z - y|^{1 - n - \alpha}\, dz \, dy. 
\end{align*}
By Fubini's Theorem, we can compute
\begin{align*} 
\int_{\R^{n}} 
& 
f(y) \, \phi(y)\cdot \int_{B_{r+\eps}(x) \setminus B_r(x)} \frac{x - z}{|x - z|} |z - y|^{1 - n - \alpha}\, dz \, dy 
\\
& = 
\int_{B_{r+\eps}(x) \setminus B_r(x)}\frac{x - z}{|x - z|}\cdot \int_{\R^{n}} f(y) \, \phi(y)\,|z - y|^{1 - n - \alpha} \, dy \, dz 
\\
& = 
\int_{r}^{r + \eps} \int_{\partial B_{\rho}(x)} \frac{x - z}{|x - z|}\cdot \int_{\R^{n}} f(y) \, \phi(y)\,|z - y|^{1 - n - \alpha} \, dy \, d \Haus{n - 1}(z) \, d \rho.
\end{align*}
By Lebesgue's Differentiation Theorem, we have
\begin{equation*}
\begin{split}
\lim_{\eps\to0}\,
\frac{1}{\eps}
\int_{\R^{n}} 
& 
f(y) \, \phi(y)\cdot \int_{B_{r+\eps}(x) \setminus B_r(x)} \frac{x - z}{|x - z|}\, |z - y|^{1 - n - \alpha}\, dz \, dy
\\
&=
\lim_{\eps\to0}\,
\frac{1}{\eps}
\int_{r}^{r + \eps} \int_{\partial B_{\rho}(x)} \frac{x - z}{|x - z|}\cdot \int_{\R^{n}} f(y) \, \phi(y)\,|z - y|^{1 - n - \alpha} \, dy \, d \Haus{n - 1}(z) \, d \rho
\\
&=
\int_{\partial B_r(x)} \frac{x - z}{|x - z|}\cdot \int_{\R^{n}} f(y) \, \phi(y)\,|z - y|^{1 - n - \alpha} \, dy \, d \Haus{n - 1}(z)
\\
&=
\int_{\R^{n}} f(y) \, \phi(y)\cdot\int_{\partial B_r(x)} \frac{x - z}{|x - z|}\,|z - y|^{1 - n - \alpha}\, d \Haus{n - 1}(z)\, dy
\\
&=
\int_{\R^{n}} f(y) \, \phi(y)\cdot\int_{\R^n} |z - y|^{1 - n - \alpha}\, dD\chi_{B_r(x)}(z)\, dy
\end{split}
\end{equation*}
for $\Leb{1}$-a.e.\ $r > 0$. 
Therefore, by~\cite{CS19}*{Theorem~3.18, equation (3.26)}, we get that 
\begin{equation} \label{eq:IBP_conv_3} 
\begin{split}
\lim_{\eps\to0}
\int_{\R^{n}} 
& 
f \,\phi \cdot \nabla^{\alpha}  h_{\eps, r, x} \, dy 
\\
& = 
\frac{\mu_{n, \alpha}}{n + \alpha - 1} 
\int_{\R^{n}} f(y) \, \phi(y) \cdot\int_{\R^{n}} |z - y|^{1 - n - \alpha} \, d D \chi_{B_{r}(x)}(z)  \, dy 
\\
& = 
\int_{\R^{n}} f \,\phi \cdot \nabla^{\alpha} \chi_{B_{r}(x)} \, dy  
\end{split}
\end{equation}
for $\Leb{1}$-a.e.\ $r > 0$. 
Finally, for the third term, we note that
\begin{equation*}
\abs*{\frac{(z - y) \cdot (\phi(z) - \phi(y)) (h_{\eps, r, x}(z) - h_{\eps, r, x}(y))}{|z - y|^{n + \alpha + 1}}} 
\le 
2\,\frac{|\phi(z) - \phi(y)|}{|z - y|^{n + \alpha}} \in L^{1}_{z}(\R^{n})
\end{equation*} 
for all $y\in\R^n$, so that
\begin{equation*}
\lim_{\eps\to0}\div_{\rm NL}^{\alpha}(h_{\eps, r, x}, \phi)(y) 
= \div^{\alpha}_{\rm NL}(\chi_{B_r(x)}, \phi)(y) 
\end{equation*}
for $\Leb{n}$-a.e.\ $y\in\R^n$ by Lebesgue's Dominated Convergence Theorem. Since
\begin{equation*}
\abs*{\div_{\rm NL}^{\alpha}(h_{\eps, r, x}, \phi)(y)} 
\le 2 \int_{\R^{n}}\frac{|\phi(z) - \phi(y)|}{|z - y|^{n + \alpha}}\,dz \in L^{1}_y(\R^{n}),
\end{equation*}
again by Lebesgue's Dominated Convergence Theorem we can compute
\begin{equation} \label{eq:IBP_conv_2} 
\lim_{\eps\to0}
\int_{\R^{n}} f \, \div_{\rm NL}^{\alpha}(h_{\eps, r, x}, \phi) \, dy 
=
\int_{\R^{n}} f \, \div^{\alpha}_{\rm NL}(\chi_{B_{r}(x)}, \phi) \, dy. 
\end{equation}
Combining~\eqref{eq:IBP_E_1}, \eqref{eq:Leibniz_rule_h_phi}, \eqref{eq:IBP_conv_1}, \eqref{eq:IBP_conv_3} and~\eqref{eq:IBP_conv_2}, we obtain~\eqref{eq:int_by_parts_u_ball}.

\smallskip

\textit{Case~2: $p \in \left (\frac{1}{1 - \alpha}, +\infty \right )$}. 
Let $(\rho_{\eps})_{\eps > 0}$ be a family of standard mollifiers (see~\cite{CS19}*{Section~3.3}) and define $f_\eps=f*\rho_\eps$ for all $\eps>0$.
By \cref{res:moll_conv_alpha_p}, we have that
\begin{equation}
\label{eq:smoothing_BV_alpha_p_step_2}
f_{\eps} 
\in 
BV^{\alpha, p}(\R^{n})
\cap 
L^{\infty}(\R^{n}) 
\cap 
C^{\infty}(\R^{n}) 
\quad\text{with}\quad
D^\alpha f_\eps = (\rho_\eps*D^\alpha f)\Leb{n}
\end{equation} 
for all $\eps>0$.
Hence, by Step~1, for each $\eps > 0$ we have
\begin{equation*} 
\int_{B_r(x)} f_{\eps} \, \div^{\alpha} \phi \, dy 
+ 
\int_{\R^{n}} f_{\eps} \, \phi\, \cdot \nabla^{\alpha} \chi_{B_r(x)}\, dy  
+ 
\int_{\R^{n}} f_{\eps} \, \div^{\alpha}_{\rm NL} (\chi_{B_r(x)}, \phi) \, dy
=
-\int_{B_r(x)} \phi \cdot \, d D^{\alpha} f_{\eps}
\end{equation*}
for $\Leb{1}$-a.e.\ $r > 0$.
We now need to study the convergence as $\eps\to0^+$ of each term of the above equality. 
By \cref{res:moll_conv_alpha_p}, we know that $D^{\alpha} f_{\eps} \weakto D^{\alpha} f$ as $\eps\to0^+$, so that
\begin{equation*}
\lim_{\eps\to0^+}
\int_{B_r(x)} \phi \cdot \, d D^{\alpha} f_{\eps} = 
\lim_{\eps\to0^+}
\int_{\R^{n}} \chi_{B_{r}(x)} \, \phi \cdot \, d D^{\alpha} f_{\eps} 
=
\int_{\R^{n}} \chi_{B_{r}(x)} \, \phi \cdot \, d D^{\alpha} f 
= 
\int_{B_{r}(x)} \phi \cdot \, d D^{\alpha} f
\end{equation*}
for $\Leb{1}$-a.e.\ $r > 0$ thanks to \cite{AFP00}*{Proposition 1.62}. 
Since 
$\div^{\alpha} \varphi \in L^{1}(\R^{n}) \cap L^{\infty}(\R^{n})$ 
by~\cite{CS19}*{Corollary~2.3} and
$\div^{\alpha}_{\rm NL}(\chi_{B_{r}(x)}, \phi) \in L^{1}(\R^{n}) \cap L^{\infty}(\R^{n})$
by~\cite{CS19}*{Lemma~2.7 and Remark~2.8}, we have
\begin{equation*}
\lim_{\eps\to0^+}
\int_{B_r(x)} f_{\eps} \, \div^{\alpha} \phi \, dy  
=
\int_{B_r(x)} f \, \div^{\alpha} \phi \, dy
\end{equation*}
and
\begin{equation*}
\lim_{\eps\to0^+}
\int_{\R^{n}} f_{\eps} \, \div^{\alpha}_{\rm NL} (\chi_{B_r(x)}, \phi) \, dy 
=
\int_{\R^{n}} f \, \div^{\alpha}_{\rm NL} (\chi_{B_r(x)}, \phi) \, dy
\end{equation*}
thanks to the fact that  $f_\eps\to f$ in~$L^p(\R^n)$ as~$\eps\to0^+$.
Finally, thanks to \cref{res:fract_grad_estimate}, we have $\nabla^{\alpha} \chi_{B_{r}(x)} \in L^{p'}(\R^{n}; \R^{n})$ for any $p \in \left (\frac{1}{1 - \alpha}, + \infty \right )$ and so
\begin{equation*}
\lim_{\eps\to0^+}
\int_{\R^{n}} f_{\eps} \, \phi\, \cdot \nabla^{\alpha} \chi_{B_r(x)}\, dy \to \int_{\R^{n}} f \, \phi\, \cdot \nabla^{\alpha} \chi_{B_r(x)}\, dy
\end{equation*}
again by the convergence $f_\eps\to f$ in~$L^p(\R^n)$ as~$\eps\to0^+$.
The conclusion thus follows.
\end{proof}

\section{Absolute continuity of the fractional variation}

\label{sec:abs_frac_var}

In this section, we prove our first main result \cref{res:abs_frac_var}.
We divide the proof into two parts, dealing with the \emph{subcritical regime}~\eqref{item:abs_frac_var_subcritical} and the \emph{supercritical regime}~\eqref{item:abs_frac_var_supercritical} separately, see \cref{res:ext_lemma328}\eqref{item:Riesz_of_BV_alpha_p} and \cref{res:abs_continuity_p} respectively.
At the end of this section, we provide two examples to show the sharpness of our result in the one-dimensional case~$n=1$.  

\subsection{The subcritical regime \texorpdfstring{$p\in\left[1,\frac n{1-\alpha}\right)$}{1<=p<n/(1-alpha)}}

Thanks to~\cite{CS19}*{Lemma~3.28}, if $f\in BV^{\alpha}(\R^n)$ then
$u = I_{1 - \alpha} f\in bv(\R^{n}) \cap L^{\frac{n}{n - 1 + \alpha}, \infty}(\R^{n})$, with $D u =D^\alpha f$ in $\M(\R^{n}; \R^{n})$.
As a consequence, we immediately deduce that
$|D^{\alpha} f| \ll \Haus{n - 1}$ for all $f \in BV^{\alpha}(\R^{n})$. 

In the following result, which is a generalization of~\cite{CS19}*{Lemma~3.28} to the present setting, we show that this phenomenon is typical of the functions belonging to $BV^{\alpha,p}(\R^n)$ in the subcritical regime $p\in\left[1,\frac n{1-\alpha}\right)$.

\begin{proposition}[Relation between $BV^{\alpha,p}(\R^n)$ and $BV^{1,p}(\R^n)$]
\label{res:ext_lemma328}
Let $\alpha \in (0,1)$, 
$p\in\left(1,\frac{n}{1-\alpha}\right)$
and
$q = \frac{np}{n - (1-\alpha)p }$.

\begin{enumerate}[(i)]

\item\label{item:Riesz_of_BV_alpha_p} 
If $f\in BV^{\alpha,p}(\R^n)$, then
$u = I_{1 - \alpha} f\in BV^{1, q}(\R^n)$ 
with 
\begin{equation*}
\|u\|_{L^q(\R^n)}
\le 
c_{n,\alpha,p}\,\|f\|_{L^p(\R^n)}
\quad\text{and}\quad
D u =D^\alpha f\ \text{in}\ \M(\R^{n}; \R^{n}).
\end{equation*}
As a consequence, we have $|D^{\alpha} f| \ll \Haus{n - 1}$ for all $f \in BV^{\alpha,p}(\R^{n})$ and the operator 
$I_{1-\alpha}\colon BV^{\alpha,p}(\R^n)\to BV^{1,q}(\R^n)$ is continuous.

%\item\label{item:frac_Lap_of_W_1_p_bv}
%If $u\in W^{1,p}(\R^n)\cap bv(\R^n)$, then
%$f=(-\Delta)^{\frac{1-\alpha}2}u\in BV^{\alpha,q}(\R^n)$ 
%with $D^\alpha f = \nabla u\,\Leb{n}$ in $\M(\R^{n}; \R^{n})$.

\item\label{item:frac_Lap_of_BV_1_p}
If $p\in\left(1,\frac{n}{n-\alpha}\right)$ and $u\in BV^{1,p}(\R^n)$, then
$f=(-\Delta)^{\frac{1-\alpha}2}u\in BV^{\alpha,p}(\R^n)$ 
with 
\begin{equation*}
\|f\|_{L^p(\R^n)}
\le c_{n,\alpha,p}\|u\|_{BV^{\alpha,p}(\R^n)}
\quad
\text{and}
\quad
D^\alpha f = D u\ \text{in}\ \M(\R^{n}; \R^{n}).
\end{equation*}
As a consequence, the operator $(-\Delta)^{\frac{1-\alpha}2}\colon BV^{1,p}(\R^n)\to BV^{\alpha,p}(\R^n)$ is continuous.

\end{enumerate}
\end{proposition}

\begin{proof}
Let $s=\frac{p}{p-1}$ and note that $r = \frac{n s}{n + (1 - \alpha)s}\in\left(1,\frac{n}{1-\alpha}\right)$.
We prove the two properties separately.

\smallskip

\textit{Proof of~\eqref{item:Riesz_of_BV_alpha_p}}.
Let $f\in BV^{\alpha,p}(\R^n)$. 
By the Hardy--Littlewood--Sobolev inequality, we immediately get that $u=I_{1-\alpha}f\in L^q(\R^n)$.
Given $\phi \in C^{\infty}_{c}(\R^{n}; \R^{n})$, we clearly have $I_{1 - \alpha} |\div \phi| \in L^{s}(\R^{n})$, because $|\div\phi|\in L^r(\R^n)$.
Hence, by Fubini Theorem, we have
\begin{equation}
\label{eq:ibp_Riesz}
\int_{\R^{n}} f\, \div^{\alpha} \phi \, dx = \int_{\R^{n}} f\, I_{1 - \alpha} \div \phi \, dx = \int_{\R^{n}} u \,\div \phi \, dx
\end{equation}
for all $\phi \in C^{\infty}_{c}(\R^{n}; \R^{n})$, proving that $D^{\alpha} f = Du$ in $\M(\R^{n}; \R^{n})$.
The remaining part of the statement in~\eqref{item:Riesz_of_BV_alpha_p} follows easily.

%\smallskip
%
%\textit{Proof of~\eqref{item:frac_Lap_of_W_1_p_bv}}.
%Let $u\in W^{1,p}(\R^n)\cap bv(\R^n)$ (and thus obviously $u\in w^{1,1}(\R^n)$).
%Since clearly $|f|\le\mathcal D^{1-\alpha} u$, from~\eqref{eq:Sobolev_extension} we immediately deduce that $f\in L^p(\R^n)$ and thus $I_{1-\alpha}f\in L^q(\R^n)$ again by the Hardy--Littlewood--Sobolev inequality.
%In particular, the composition map
%\begin{equation*}
%T\colon 
%W^{1,p}(\R^n)
%\overset{(-\Delta)^{\frac{1-\alpha}2}}{\xrightarrow{\hspace*{3.5em}}}
%L^p(\R^n)
%\overset{I_{1-\alpha}}{\xrightarrow{\hspace*{3em}}}
%L^q(\R^n)
%\end{equation*}
%is continuous. 
%Moreover, we have $Tg=g$ for all $g\in \mathcal S(\R^n)$ by applying the Fourier transform (see~\cite{LX21}*{Lemma~2.3} for instance).
%A plain density argument thus proves that $I_{1-\alpha}f(x)=Tu(x)=u(x)$ for a.e.\ $x\in\R^n$. 
%Since $W^{1,p}(\R^n)\cap bv(\R^n)\subset L^{\frac n{n-1}}(\R^n)\cap L^{\frac{np}{n-p}}(\R^n)\subset L^q(\R^n)$ with continuous inclusions by the Gagliardo--Nirenberg--Sobolev embedding and by interpolation, we deduce that $I_{1-\alpha}f=Tu=u$ in~$L^q(\R^n)$.
%Therefore, arguing as before, we get~\eqref{eq:ibp_Riesz} for all $\phi\in C^\infty_c(\R^n;\R^n)$ and the conclusion follows.

\smallskip

\textit{Proof of~\eqref{item:frac_Lap_of_BV_1_p}}.
Let $p\in\left(1,\frac{n}{n-\alpha}\right)$ and $u\in BV^{1,p}(\R^n)$.
Since $p<\frac n{n-\alpha}$, we can apply \cref{res:BV_alpha_p_in_S_beta_q} to get that $BV^{1,p}(\R^n)\subset S^{1-\alpha,p}(\R^n)$ with continuous inclusion, so that $f=(-\Delta)^{\frac{1-\alpha}2}u\in L^p(\R^n)$ (thanks to the identification $S^{1-\alpha,p}(\R^n)=L^{1-\alpha,p}(\R^n)$ following from~\cite{BCCS20}*{Corollary~2.1}, also see the discussion in~\cite{BCCS20}*{Section~2.1}) and thus $I_{1-\alpha}f\in L^q(\R^n)$ by the Hardy--Littlewood--Sobolev inequality.
Since $p<\frac n{n-\alpha}$, we also have that $p<q<\frac n{n-1}$ and thus $BV^{1,p}(\R^n)\subset L^q(\R^n)$ with continuous inclusion by \cref{res:alvino}.
Hence $u\in L^q(\R^n)$ and we can now claim that $I_{1-\alpha}f=u$ in $L^q(\R^n)$.
Indeed, this is easily verified if $u\in C^\infty_c(\R^n)$ by applying the Fourier transform (see~\cite{LX21}*{Lemma~2.3} for instance), so that the claim follows by a plain approximation argument.
Therefore, by applying Fubini Theorem again, we can write~\eqref{eq:ibp_Riesz} and prove that $D^\alpha f=Du$ in $\M (\R^n;\R^n)$, reaching the conclusion. 
\end{proof}

\begin{remark}[About \cref{res:ext_lemma328}\eqref{item:frac_Lap_of_BV_1_p}]
The validity of \cref{res:ext_lemma328}\eqref{item:frac_Lap_of_BV_1_p} when $p=1$ was already proved in~\cite{CS19}*{Lemma~3.28}.
We also refer the reader to~\cite{KS21}*{Proposition~3.1}, in which the authors prove that, if $u\in W^{1,p}(\R^n)$ for some $p\in[1,+\infty]$, then $f=(-\Delta)^{\frac{1-\alpha}{2}}u\in S^{\alpha,p}(\R^n)$ with $\nabla^\alpha f=\nabla u$ in $L^p(\R^n;\R^n )$.
\end{remark}

\subsection{The supercritical regime \texorpdfstring{$p\in\left[\frac{n}{1-\alpha},+\infty\right]$}{}}

We now focus on the absolute continuity property of the fractional variation with respect to the Hausdorff measure for functions belonging to $BV^{\alpha,p}(\R^n)$ in the supercritical regime
$p\in\left[\frac{n}{1-\alpha},+\infty\right]$. 
The crucial tool in this case is provided by the following important consequence of \cref{res:int_by_parts_E_ball}, which extends~\cite{CS19}*{Theorem~5.3} to the present setting.

\begin{theorem}[Decay estimates for $BV^{\alpha,p}$ functions for $p>\frac 1{1-\alpha}$]
\label{res:decay_estimate_p} 
Let $\alpha \in (0,1)$ and
$p\in\left(\frac{1}{1-\alpha},+\infty\right]$.
There exist two constants $A_{n,\alpha,p}, B_{n,\alpha,p} > 0$, depending on $n$, $\alpha$ and $p$ only, with the following property. 
If $f \in BV^{\alpha,p}(\mathbb{R}^n)$ then,
for $|D^{\alpha} f|$-a.e.\ $x \in \R^n$, there exists $r_x > 0$ such that
\begin{equation}
\label{eq:decay_D_alpha_u_B_1_prime} 
|D^{\alpha} f|(B_r(x)) 
\le 
A_{n,\alpha,p} 
\|f\|_{L^p(\R^n)}\,
r^{\frac{n}{q} - \alpha} 
\end{equation}
and
\begin{equation}
\label{eq:decay_D_alpha_u_B_2_prime}  
|D^{\alpha} (f \chi_{B_r(x)})|(\mathbb{R}^{n}) 
\le 
B_{n,\alpha,p} 
\|f\|_{L^p(\R^n)}\,
r^{\frac{n}{q} - \alpha} 
\end{equation}
for all $r \in (0, r_x)$, 
where $q\in[1,+\infty)$ is such that $\frac1p+\frac1q=1$.
\end{theorem}

\begin{proof}%[Proof of \cref{res:decay_estimate_p}]
Since $f \in BV^{\alpha,p}(\R^n)$, by the Polar Decomposition Theorem for Radon measures there exists a Borel vector valued function $\sigma_{f}^{\alpha}\colon\R^n\to\R^n$ such that 
\begin{equation}
\label{eq:polar_decomp_D_alpha}
D^{\alpha} f 
= 
\sigma_{f}^{\alpha}\,|D^{\alpha} f| 
\quad\text{with}\quad 
|\sigma_{f}^{\alpha}(x)| = 1\ 
\text{for} \ 
|D^{\alpha} f|\text{-a.e.}\ 
x \in \R^{n}.
\end{equation}
We divide the proof into two steps, dealing with the two estimates separately.

\smallskip

\textit{Step~1: proof of~\eqref{eq:decay_D_alpha_u_B_1_prime}}. 
Let $\sigma_{f}^{\alpha}: \R^{n} \to \R^{n}$ be as in~\eqref{eq:polar_decomp_D_alpha} and let $x\in\R^n$ be such that $|\sigma_{f}^{\alpha}(x)| = 1$.
Given $r>0$, we define the vector field $\phi\colon\R^n\to\R^n$ by setting
\begin{equation} 
\label{eq:phi_x_r_def}
\phi_{x, r}(y) = 
\begin{cases} 
\sigma_{f}^{\alpha}(x) 
& \text{if} \,  y \in B_{r}(x),\\
\sigma_{f}^{\alpha}(x) \left (2 - \frac{|y - x|}{r} \right ) 
& \text{if} \, y \in B_{2r}(x) \setminus B_{r}(x),\\
0 
& \text{if} \, y \notin B_{2 r}(x),
\end{cases}
\end{equation}
for all $y\in\R^n$.
We clearly have that
$\phi_{x, r} \in \Lip_{c}(\R^{n}; \R^{n})$ 
with
$\|\phi\|_{L^{\infty}(\R^{n}; \R^{n})} \le 1$.
Thus, on the one hand, we can find $r_x\in(0,1)$ such that
\begin{equation} 
\label{eq:decay_estimate_1_p} 
\int_{B_r(x)} \phi_{x, r}(y) \cdot\, d D^{\alpha} f(y) 
= 
\int_{B_{r}(x)} \sigma_{f}^{\alpha}(x) \cdot \sigma_{f}^{\alpha}(y) \, d |D^{\alpha}f|(y) 
\ge 
\frac{1}{2} |D^{\alpha} f|(B_r(x)) 
\end{equation}
for all $r\in(0,r_x)$. 
On the other hand, by~\eqref{eq:int_by_parts_u_ball} we can write
\begin{equation}
\label{eq:decay_estimate_1.5_p}
\begin{split} 
\int_{B_r(x)} \phi_{x, r}\cdot\, d D^{\alpha} f 
&\le 
\abs*{\int_{B_{r}(x)} f \, \div^{\alpha} \phi_{x, r} \, dy} 
+
\abs*{ \int_{\R^{n}} f \, \phi_{x, r} \cdot \nabla^{\alpha} \chi_{B_r(x)} \, dx}\\
&\quad+
\abs*{ \int_{\R^{n}} f \, \div^{\alpha}_{\rm NL} (\chi_{B_r(x)}, \phi_{x, r}) \, dy } 
\end{split}
\end{equation}
for $\Leb{1}$-a.e.\ $r\in(0,r_x)$. 
We now estimate the three terms in the right-hand side of~\eqref{eq:decay_estimate_1.5_p} separately.
For the first term, recalling the definition of $\phi_{x, r}$ in~\eqref{eq:phi_x_r_def}, we have
\begin{align*}
\abs*{\int_{B_{r}(x)} f \, \div^{\alpha} \phi_{x, r} \, dy} 
& \le 
\mu_{n, \alpha} 
\int_{B_{r}(x)} |f(y)| 
\int_{\R^{n} \setminus B_{r}(x)} \frac{|\varphi_{x, r}(z) - \sigma_{f}^{\alpha}(x)|}{|z - y|^{n + \alpha}} \, dz 
dy 
\\
& \le 
\mu_{n, \alpha} 
\|f\|_{L^{p}(B_{r}(x))} 
\left ( 
\int_{B_{r}(x)} 
\left ( 
\int_{\R^{n} \setminus B_{r}(x)} \frac{|\varphi_{x, r}(z) - \sigma_{f}^{\alpha}(x)|}{|z - y|^{n + \alpha}} 
\, dz 
\right )^{q} 
dy 
\right)^{\frac{1}{q}} 
\\
& \le 
2 \mu_{n, \alpha} 
\|f\|_{L^{p}(B_{r}(x))} 
\,
r^{\frac{n}{q} - \alpha} 
\left ( 
\int_{B_{1}} 
\left ( 
\int_{\R^{n} \setminus B_{1}} \frac{1}{|z - y|^{n + \alpha}} \, dz \right )^{q} 
dy 
\right)^{\frac{1}{q}}.
\end{align*}
After some elementary computations, we get
\begin{equation}
\label{eq:1st_term_omissis}
\left(
\int_{\R^{n} \setminus B_{1}(-y)} 
\frac{1}{|z|^{n + \alpha}} \, dz 
\right )^q \, dy 
\le
C_{n,\alpha}
\int_{0}^{1} 
\frac{t^{n - 1}}{(1 - t)^{\alpha q}} 
\, dt
\end{equation}
for some constant $C_{n,\alpha}>0$ depending only on $n$ and $\alpha$.
Note that the integral appearing in the right-hand side of~\eqref{eq:1st_term_omissis} converges if and only if $\alpha q < 1$, that is, $p > \frac{1}{1 - \alpha}$. 
We thus get
\begin{equation}
\label{eq:mattacchione_I}
\abs*{\int_{B_{r}(x)} f \, \div^{\alpha} \phi_{x, r} \, dy}
\le
C_{n,\alpha,q}\,
\|f\|_{L^{p}(\R^n)} 
\,
r^{\frac{n}{q} - \alpha} 	
\end{equation}
for some constant $C_{n,\alpha,q}>0$ depending only on $n$, $\alpha$ and~$q$.
For the second term in the right-hand side of~\eqref{eq:decay_estimate_1.5_p}, we have
\begin{equation} \label{eq:estimate_u_nabla_alpha_B_r_x}
\left | \int_{\R^{n}} f \phi_{x, r} \cdot \nabla^{\alpha} \chi_{B_{r}(x)} \, dy \right | \le 
\|f\|_{L^{p}(B_{2r}(x))} \|\nabla^{\alpha}\chi_{B_{r}(x)}\|_{L^{q}(\R^{n}; \R^{n})} 
\le 
C_{n, \alpha,q} 
\|f\|_{L^{p}(B_{2 r}(x))} \, 
r^{\frac{n}{q} - \alpha}
\end{equation}
thanks to \cref{res:fract_grad_estimate}, for some constant $C_{n, \alpha,q}>0$ depending only on $n$, $\alpha$ and~$q$.
Finally, observing that 
$
\phi_{x, r}(x + r y) 
= 
\varphi_{0, 1}(y)	
$
for all $y\in\R^n$, a simple change of variables gives
\begin{equation}
\label{eq:div_NL_alpha_prime}
\|\div_{\rm NL}^{\alpha}(\chi_{B_{r}(x)}, \varphi_{x, r})\|_{L^{q}(\R^{n})} 
= r^{\frac nq - \alpha}\,
\|\div_{\rm NL}^{\alpha}(\chi_{B_{1}}, \varphi)\|_{L^{q}(\R^{n})}.
\end{equation}
Thus, for the third and last term in the right-hand side of~\eqref{eq:decay_estimate_1.5_p}, we have 
\begin{equation}
\label{eq:mattacchione_III}
\abs*{
\int_{\R^{n}} f\, \div_{\rm NL}^{\alpha}(\chi_{B_{r}(x)}, \varphi_{x, r}) \, dy
}
\le
C_{n, \alpha, q}\, 
\|f\|_{L^{p}(\R^{n})} r^{\frac{n}{q} - \alpha},
\end{equation}
where $C_{n, \alpha, q}=\|\div_{\rm NL}^{\alpha}(\chi_{B_{1}}, \varphi)\|_{L^{q}(\R^{n})}$ (which is finite thanks to~\cite{CS19}*{Lemma~2.7 and Remark~2.8}).
Combining~\eqref{eq:decay_estimate_1_p} with~\eqref{eq:decay_estimate_1.5_p}, \eqref{eq:mattacchione_I}, \eqref{eq:estimate_u_nabla_alpha_B_r_x} and~\eqref{eq:mattacchione_III}, we get~\eqref{eq:decay_D_alpha_u_B_1_prime} with a simple continuity argument.

\smallskip

\textit{Step~2: proof of~\eqref{eq:decay_D_alpha_u_B_2_prime}}.
Let $x\in\R^n$ be such that $|\sigma_{u}^{\alpha}(x)| = 1$.
Given $\phi \in \Lip_{c}(\R^{n}; \R^{n})$ such that $\|\phi\|_{L^{\infty}(\R^{n}; \R^{n})} \le 1$, from~\eqref{eq:int_by_parts_u_ball} we deduce that
\begin{align*} 
\abs*{\int_{B_r(x)} f \, \div^{\alpha} \phi \, dy} 
&\le 
|D^{\alpha} f|(B_r(x)) 
+ 
\|f\|_{L^{p}(\R^{n})} \, 
\|\nabla^{\alpha} \chi_{B_r(x)}\|_{L^{q}(\R^n; \R^{n})}\\
&\quad+ 
\|f\|_{L^{p}(\R^{n})} \|\div^{\alpha}_{\rm NL}(\chi_{B_r(x)}, \phi)\|_{L^{q}(\R^{n})}
\end{align*}
for $\Leb{1}$-a.e.\ $r\in(0,r_x)$. Exploiting~\eqref{eq:decay_D_alpha_u_B_1_prime}, \eqref{eq:ball_fract_grad_estimate} and~\eqref{eq:div_NL_alpha_prime}, we conclude that
\begin{equation*} 
|D^{\alpha} ( f\chi_{B_r(x)} )|(\R^{n}) 
\le 
C_{n, \alpha, q} \, r^{\frac{n}{q} - \alpha}
\end{equation*}
for $\Leb{1}$-a.e.\ $r\in(0,r_x)$, where $C_{n, \alpha, q}>0$ is a constant depending only on $n$, $\alpha$ and~$q$. Inequality~\eqref{eq:decay_D_alpha_u_B_2_prime} thus follows for all $r\in(0,r_x)$ by a simple continuity argument.
\end{proof}

Thanks to \cref{res:decay_estimate_p} and extending~\cite{CS19}*{Corollary~5.4} to the present setting, we are now ready to state and prove the following absolute continuity property of the fractional variation for $BV^{\alpha,p}$ functions with $p\in\left(\frac{1}{1-\alpha},+\infty\right]$. Note that the result below is truly interesting only for $p\in\left [\frac{n}{1-\alpha},+\infty\right]$, due to \cref{res:abs_frac_var}\eqref{item:abs_frac_var_subcritical} (see also \cref{res:ext_lemma328}) and the fact that
\begin{equation*}
n - \alpha - \frac{n}{p} \ge n - 1
\iff 
p \ge \frac{n}{1 - \alpha}.
\end{equation*}

\begin{corollary}[$|D^{\alpha} f|\ll \Haus{\frac nq -\alpha}$ for $p>\frac1{1-\alpha}$] 
\label{res:abs_continuity_p} 
Let $\alpha \in (0,1)$ and
$p\in\left(\frac{1}{1-\alpha},+\infty\right]$.
If $f\in BV^{\alpha,p}(\R^n)$, then there exists a $|D^\alpha f|$-negligible set $Z_f^{\alpha,p}\subset\R^n$ such that
\begin{equation}
\label{eq:abs_cont_estimate}
|D^{\alpha} f| 
\le 
2^{\frac nq -\alpha}\,
\frac{A_{n,\alpha,p}}{\omega_{\frac nq-\alpha}} \,
\|f\|_{L^p(\R^{n})} \, 
\Haus{\frac nq - \alpha} \res \R^{n} \setminus Z_f^{\alpha,p},
\end{equation}
where $A_{n,\alpha,p}$ is as in~\eqref{eq:decay_D_alpha_u_B_1_prime} and $q\in[1,+\infty)$ is such that $\frac1p+\frac1q=1$.
\end{corollary}

\begin{proof}
By \cref{res:decay_estimate_p}, there exists a set $Z_{f}^{\alpha,p}\subset\R^n$ such that $|D^{\alpha}f|(Z_{f}^{\alpha,p}) = 0$ and \eqref{eq:decay_D_alpha_u_B_1_prime} holds for any $x \notin Z_{f}^{\alpha,p}$. 
Thanks to the Borel regularity of the Radon measure $|D^{\alpha} f|$, we can assume that $Z_{f}^{\alpha,p}$ is a Borel set without loss of generality. 
Hence, for all 
$x\in\R^{n}\setminus Z_{f}^{\alpha,p}$, 
we have
\begin{equation*} 
\Theta^{*}_{\frac nq - \alpha}(|D^{\alpha} f|, x) 
= 
\limsup_{r \to 0^+} \frac{|D^{\alpha}f|(B_{r}(x))}{\omega_{\frac nq - \alpha} r^{\frac nq - \alpha}} 
\le 
\frac{A_{n, \alpha,p}}{\omega_{\frac nq - \alpha}}\, 
\|f\|_{L^{p}(\R^{n})}.
\end{equation*}
Inequality~\eqref{eq:abs_cont_estimate} thus follows from~\cite{AFP00}*{Theorem 2.56}.
\end{proof}

\begin{remark}[The case $n=1$ and $p=\frac1{1-\alpha}$]
Note that \cref{res:abs_continuity_p} covers the supercritical regime $p\in\left [\frac{n}{1-\alpha},+\infty\right]$ for $n\ge2$, while for $n=1$ the boundary case $p=\frac1{1-\alpha}$ is missing. 
However, if $n=1$ and $p=\frac1{1-\alpha}$, then $q=\frac p{p-1}=\frac1\alpha$ and so $\Haus{\frac nq-\alpha}=\Haus{0}$, so that $|D^\alpha f|\ll\Haus{0}$ for all $f\in BV^{\alpha,\frac 1{1-\alpha}}(\R)$ trivially.
We do not know if this result is sharp.
\end{remark}

\begin{remark}[The limit as $\alpha \to 1^-$]
It is somewhat interesting to observe that \cref{res:abs_continuity_p} still holds true if we send $\alpha \to 1^-$. Indeed, such a limit case would apply only to functions $f \in BV^{1, \infty}(\R^n)$, for which it is well known (see \cite{AFP00}*{Theorem 3.77, Theorem 3.78 and equation (3.90)}, for instance) that
\begin{equation*}
|Df| \le 2 \|f\|_{L^{\infty}(\R^n)} \Haus{n-1} \res J_f,
\end{equation*}
where $J_f$ is the jump set, so that $Z_f^{1, \infty}$ could be any $|Df|$-negligible subset of $\R^n \setminus J_f$.
\end{remark}

\subsection{Two examples in one dimension}

We conclude this section by discussing the optimality of the absolute continuity properties of the fractional variation stated in \cref{res:abs_frac_var} in the one-dimensional case $n=1$. 

We begin with the following example, which is borrowed from~\cite{CS19}*{Theorem 3.26}.

\begin{example}[\cref{res:ext_lemma328}\eqref{item:Riesz_of_BV_alpha_p} is sharp for $n=1$]
\label{exa:f_alpha}
Let $\alpha \in (0, 1)$ and consider
\begin{equation*}
f_{\alpha}(x) = \mu_{1, - \alpha} \left ( |x|^{\alpha - 1} \sgn{x} - |x - 1|^{\alpha - 1} \sgn(x - 1) \right ),
\quad
x\in\R.
\end{equation*}
By~\cite{CS19}*{Theorem 3.26}, we have $f_{\alpha} \in BV^{\alpha}(\R)$ with
$D^{\alpha} f_{\alpha} = \delta_{0} - \delta_{1}$.
Moreover, by~\cite{CS19-2}*{Theorem 3.8 and Remark 3.9}, we have 
$f_{\alpha} \in L^{\frac{1}{1 - \alpha}, \infty}(\R) \setminus L^{\frac{1}{1 - \alpha}, q}(\R)$ for all $q \ge 1$. In particular, since $f_{\alpha} \in L^{1}(\R) \cap L^{\frac{1}{1 - \alpha}, \infty}(\R)$, by interpolation we get that $f_{\alpha} \in L^{p}(\R)$ for all 
$p \in \left [1, \frac{1}{1 - \alpha} \right )$.
Hence $f_\alpha\in BV^{\alpha,p}(\R)$ for all
$p \in \left [1, \frac{1}{1 - \alpha} \right )$
with $|D^\alpha f_\alpha|\not\ll\Haus{\eps}$ for all $\eps > 0$.
This proves that the absolute continuity property of the fractional variation stated in \cref{res:abs_frac_var}\eqref{item:abs_frac_var_subcritical} is sharp for $n=1$. 
\end{example}

We now prove the following result, which combines the properties of the function $f_\alpha$ introduced in \cref{exa:f_alpha} with the decay properties of a finite Radon measure.

\begin{proposition}[The function $u_\alpha=f_\alpha*\nu$]
\label{res:u_alpha}
Let $\alpha\in(0,1)$, let $f_\alpha$ be as in \cref{exa:f_alpha}, and let $\nu\in\mathscr M(\R)$.
Then we have 
\begin{equation*}
u_\alpha=f_\alpha*\nu \in BV^{\alpha,p}(\R) 
\qquad
\text{for all}\ p \in \left [1, \frac{1}{1 - \alpha} \right ),	
\end{equation*}
with 
\begin{equation} \label{eq:fract_der_nu}
D^{\alpha} u_{\alpha} = \nu - (\tau_{1})_{\#} \nu,
\end{equation}
where 
$\tau_{x}(y) = y + x$ for all $x,y\in\R$. In addition, if there exist $C, \eps > 0$ such that
\begin{equation}
\label{eq:mu_radius}
\nu\big((x-r,x+r)\big)
\le
Cr^\eps
\quad\text{for all $x\in\R$ and $r>0$},
\end{equation}
then 
\begin{equation}
\label{eq:u_alpha_integrability}
u_\alpha \in BV^{\alpha,p}(\R)
\quad
\text{for all}\
p\in
\begin{cases}
\left [1, \frac{1 - \eps}{1 - \alpha - \eps}\right )
& 
\text{if}\ \eps \in (0, 1 - \alpha),\\[2mm]
[1,+\infty)
& 
\text{if}\ \eps = 1 - \alpha,\\[2mm]
[1,+\infty]
& 
\text{if}\ \eps \in (1 - \alpha,1].\\[2mm]
\end{cases}
\end{equation}
\end{proposition}

\begin{proof}
We divide the proof into two steps.

\smallskip

\textit{Step~1}.
Let $\nu\in\mathscr M(\R)$. We start by showing that $u_{\alpha}\in BV^{\alpha, p}(\R)$ for all 
$p \in \left [1, \frac{1}{1 - \alpha} \right )$ and that it satisfies \eqref{eq:fract_der_nu}.
Indeed, by Young's inequality (for Radon measures) we can estimate
\begin{equation*}
\|u_{\alpha}\|_{L^{1}(\R)} 
\le 
\|f_{\alpha}\|_{L^{1}(\R)} |\nu|(\R).
\end{equation*}
Moreover, thanks to the translation invariance of $\div^{\alpha}$ and exploiting the explicit expression of $f_\alpha$ given in \cref{exa:f_alpha}, we can write 
\begin{align*}
\int_{-\infty}^{\infty} u_{\alpha}(x) \,\div^{\alpha} \varphi(x) \, dx 
& = 
\int_{- \infty}^{\infty} \int_{- \infty}^{\infty} f_{\alpha}(x - y) \, \div^{\alpha} \varphi (x) \, d \nu(y) \, dx
\\ 
& = 
\int_{- \infty}^{\infty} \int_{- \infty}^{\infty} f_{\alpha}(x - y)\, \div^{\alpha} \varphi(x) \, dx \, d \nu(y) 
\\
& = 
- \int_{- \infty}^{\infty} \int_{- \infty}^{\infty} \varphi(x + y) \, d \left ( \delta_{0}(x) - \delta_{1}(x) \right ) \, d \nu(y) \\
& = 
- \int_{- \infty}^{\infty} \left ( \varphi(y) - \varphi(y + 1) \right ) \, d \nu(y)
\end{align*}
for all $\phi\in C^\infty_c(\R)$.
Thus $u_{\alpha} \in BV^{\alpha, 1}(\R)$ with
$D^{\alpha} u_{\alpha} = \nu - (\tau_{1})_{\#} \nu$.
In addition, by Jensen's inequality and Tonelli's Theorem we can estimate
\begin{align*}
\int_{- \infty}^{\infty} |u_{\alpha}(x)|^{p} \, dx 
& \le \int_{- \infty}^{\infty}  
|\nu|(\R)^{p - 1}
\int_{- \infty}^{\infty} |f_{\alpha}(x - y)|^{p} \, d |\nu|(y) \, dx 
= |\nu|(\R)^{p} \,\|f_{\alpha}\|_{L^{p}(\R)}^p 
< + \infty
\end{align*}
for all $p\in\left[1,\frac1{1-\alpha}\right)$, thanks to the integrability properties of $f_\alpha$ given in \cref{exa:f_alpha}.

\smallskip

\textit{Step~2}.
We prove that~\eqref{eq:mu_radius} implies~\eqref{eq:u_alpha_integrability}.
To this aim, let $\delta > 0$ and $q=\frac p{p-1}$.
Since 
$|f_{\alpha}| = |f_{\alpha}|^{\frac{\delta}{q}} |f_{\alpha}|^{1 - \frac{\delta}{q}}$,
by H\"older's inequality we get
\begin{align*}
|u_{\alpha}(x)|^{p} 
& \le 
\left ( \int_{- \infty}^{\infty} |f_{\alpha}(x - y)|^{\frac{\delta}{q}}\, |f_{\alpha}(x - y)|^{1 - \frac{\delta}{q}} \, d |\nu|(y) \right )^{p} 
\\
& \le 
\left ( \int_{- \infty}^{\infty} |f_{\alpha}(x - y)|^{\delta} \, d |\nu|(y) \right )^{\frac{p}{q}} 
\left ( \int_{- \infty}^{\infty} |f_{\alpha}(x - y)|^{p \left (1 - \frac{\delta}{q} \right )} \, d |\nu|(y) \right )
\end{align*}
for $x\in\R$.
We now recall the explicit expression of $f_\alpha$ in \cref{exa:f_alpha} and write
\begin{equation}
\label{eq:u_alpha_splitting}
\begin{split}
\int_{- \infty}^{\infty} |f_{\alpha}(x - y)|^{\delta} \, d |\nu|(y) 
& = 
\int_{\left (- \infty, x - \frac{3}{2} \right ) \cup \left (x + \frac{1}{2}, \infty \right )} |f_{\alpha}(x - y)|^{\delta} \, d |\nu|(y) 
\\
& \quad + \sum_{j = 1}^{\infty} \int_{I_{j}\left (x, \frac{1}{2} \right ) \cup I_{j}\left (x - 1, \frac{1}{2} \right )} |f_{\alpha}(x - y)|^{\delta} \, d |\nu|(y),
\end{split}
\end{equation}
where we have set
\begin{equation*}
I_{j}(x, r) = (x - r^{j}, x + r^{j}) \setminus (x - r^{j + 1}, x + r^{j + 1}) 
\end{equation*}
for all $x\in\R$, $r\in(0,1)$ and $j\in\N$ for brevity.
Now, on the one hand, if 
$y \in \left (- \infty, x - \frac{3}{2} \right ) \cup \left (x + \frac{1}{2}, \infty \right )$, 
then $x - y \in \left (- \infty, - \frac{1}{2} \right ) \cup \left (\frac{3}{2}, \infty \right )$, so that
\begin{equation*}
|f_{\alpha}(x - y)| \le \mu_{1, - \alpha} \left ( 2^{1 - \alpha} + 2^{1 - \alpha} \right ) = \mu_{1, - \alpha} 2^{2 - \alpha}
\end{equation*}
for all 
$y \in \left (- \infty, x - \frac{3}{2} \right ) \cup \left (x + \frac{1}{2}, \infty \right )$.
Therefore, we can estimate
\begin{equation}
\label{eq:u_alpha_split_est1}
\int_{\left (- \infty, x - \frac{3}{2} \right ) \cup \left (x + \frac{1}{2}, \infty \right )} |f_{\alpha}(x - y)|^{\delta} \, d |\nu|(y) 
\le 
\left (\mu_{1, - \alpha} 2^{2 - \alpha} \right )^{\delta} |\nu|(\R)
\end{equation}
for all $x\in\R$.
On the other hand, 
for all $x \in \R$ and $j \in \N$, 
we have
\begin{align}
\int_{I_{j}\left (x, \frac{1}{2} \right)} |f_{\alpha}(x - y)|^{\delta} \, d |\nu|(y) 
& \le 
\mu_{1, - \alpha}^{\delta} \int_{I_{j}\left (x, \frac{1}{2} \right)} \left ( |x - y|^{\alpha - 1} + |x - y - 1|^{\alpha - 1} \right )^{\delta} \, d |\nu|(y) 
\nonumber\\
& \le \mu_{1, - \alpha}^{\delta} \left ( 2^{(j + 1)(1 - \alpha)} + \left ( 1 - 2^{-j} \right )^{\alpha - 1} \right )^{\delta} |\nu|\left ( (x - 2^{-j}, x + 2^{j} ) \right ) 
\nonumber\\
& \le \label{eq:u_alpha_split_est2} 
\mu_{1, - \alpha}^{\delta} \left ( 2^{(j + 1)(1 - \alpha)} + 2^{1 - \alpha} \right )^{\delta} C \, 2^{- j \eps}.
\end{align}
Reasoning analogously, we obtain
\begin{equation}
\label{eq:u_alpha_split_est3}
\int_{I_{j}\left (x- 1, \frac{1}{2} \right)} |f_{\alpha}(x - y)|^{\delta} \, d |\nu|(y) \le C \mu_{1, - \alpha}^{\delta} \left ( 2^{(j + 1)(1 - \alpha)} + 2^{1 - \alpha} \right )^{\delta} 2^{- j \eps}
\end{equation}
for all $x \in \R$ and $j \in \N$.
Therefore, inserting~\eqref{eq:u_alpha_split_est1}, \eqref{eq:u_alpha_split_est2} and~\eqref{eq:u_alpha_split_est3} in~\eqref{eq:u_alpha_splitting}, we conclude that 
\begin{equation} 
\label{eq:f_alpha_delta_bound}
\int_{- \infty}^{\infty} |f_{\alpha}(x - y)|^{\delta} \, d |\nu|(y) 
\le 
C_{\alpha, \eps, \delta}
\end{equation}
for all $x\in\R$, where $C_{\alpha, \eps, \delta}>0$ is constant depending on~$\alpha$, $\eps$, and~$\delta$ which is finite provided that we choose $\delta <\frac\eps{1-\alpha}$, as we are assuming from now on.
We thus get
\begin{align*}
\int_{- \infty}^{\infty} |u_{\alpha}(x)|^{p} \, dx 
& \le 
C_{\alpha, \eps, \delta}^{p - 1}\, 
\int_{- \infty}^{\infty} 
\int_{- \infty}^{\infty} 
|f_{\alpha}(x - y)|^{p \left (1 - \frac{\delta}{q} \right )} \, d |\nu|(y) \, dx 
\\
& = 
C_{\alpha, \eps, \delta}^{p - 1} 
\,
|\nu|(\R) 
\int_{- \infty}^{\infty} |f_{\alpha}(x)|^{p \left (1 - \frac{\delta}{q} \right )} \, dx.
\end{align*}
Now, recalling \cref{exa:f_alpha}, we immediately see that
\begin{equation}
\label{eq:f_alpha_int_conditions}
\int_{- \infty}^{\infty} |f_{\alpha}(x)|^{p \left (1 - \frac{\delta}{q} \right )} \, dx < +\infty
\iff
\begin{cases}
p < \frac{1}{(1 - \alpha) (1 - \delta)} - \frac{\delta}{1 - \delta} = \frac{1 - \delta + \alpha \delta}{(1 - \alpha) (1 - \delta)}, 
\\[3mm]
p > \frac{1}{(2 - \alpha) (1 - \delta)} - \frac{\delta}{1 - \delta} = \frac{1 - 2 \delta + \alpha \delta}{(2 - \alpha) (1 - \delta)}.
\end{cases}
\end{equation}
Since one easily recognizes that
\begin{equation*}
 \frac{1 - 2 \delta + \alpha \delta}{(2 - \alpha) (1 - \delta)} < 1  
\quad 
\text{for all}\ \alpha \in (0, 1) \text{ and } \delta > 0,
\end{equation*}
the second condition on~$p$ in~\eqref{eq:f_alpha_int_conditions} can be dropped.
As for the first condition on~$p$ in~\eqref{eq:f_alpha_int_conditions}, 
it is readily seen that
\begin{equation*}
\eps \in (0, 1 - \alpha)
\implies
\delta < \frac{\eps}{1 - \alpha} < 1
\implies
p \in \left [1, \frac{1 - \eps}{1 - \alpha - \eps}\right)
\end{equation*}
and, similarly,
\begin{equation*}
\eps \in [1 - \alpha, 1]
\implies
\delta (1 - \alpha) < \eps\
\text{for all}\ 
\delta\in(0,1)
\implies
p \in [1, +\infty).
\end{equation*}
Finally, in the case $\eps \in (1 - \alpha, 1]$, we exploit \eqref{eq:f_alpha_delta_bound} for $\delta = 1$ in order to conclude that
\begin{equation*}
|u_{\alpha}(x)| 
\le 
\int_{- \infty}^{\infty} |f_{\alpha}(x - y)| \, d |\nu|(y) 
= 
C_{\alpha, \eps} 
< 
+\infty
\end{equation*}
for all $x\in\R$, which implies that $u_{\alpha} \in L^{\infty}(\R)$.
The conclusion thus follows.
\end{proof}

Thanks to \cref{res:u_alpha}, we can now give the following example.

\begin{example}[\cref{res:abs_continuity_p} is sharp for $n=1$]
Let $\alpha\in(0,1)$ and let $\nu$ and $u_\alpha$ be as in \cref{res:u_alpha}.
By \cite{F14}*{Corollary 4.12}, there exists a compact set $K\subset\R$ such that $\nu = \Haus{\eps} \res K$, so that $D^{\alpha} u_{\alpha}\not\ll\Haus{s}$ for all $s > \eps$. 
Now we observe that, by~\eqref{eq:u_alpha_integrability}, we have the following situations:
\begin{itemize}

\item
if $\eps \in (0, 1 - \alpha)$, then
$p < \frac{1 - \eps}{1 - \alpha - \eps} < \frac{1}{1 - \alpha - \eps}$
and thus 
$\eps > \frac{1}{q} - \alpha$;

\item
if $\eps = 1 - \alpha$, then $p \in [1, +\infty)$ and thus $\eps > \frac{1}{q} - \alpha$;

\item
if $\eps \in (1 - \alpha, 1]$, then $p \in [1, +\infty]$ and so, for $p=+\infty$, if $s > 1 - \alpha$ then we can take $\eps \in (1 - \alpha, s)$.
\end{itemize}
Therefore, the absolute continuity property of the fractional variation stated in \cref{res:abs_frac_var}\eqref{item:abs_frac_var_supercritical} is sharp for $n=1$.
\end{example}

\section{Fractional capacity and precise representative}

\label{sec:capa_prec_rep}

In this last section, we study the fractional capacity and the existence of the precise representatives of $BV^{\alpha,p}$ functions. 

\subsection{The \texorpdfstring{$(\alpha,p)$}{(alpha,p)}-capacity}

We begin with the definition of \emph{fractional capacity}, see~\cite{AH96}*{Chapter~2}.
For the classical integer case $\alpha=1$, we also refer the reader to~\cite{EG15}*{Sections~4.7 and~5.6.3}, \cite{HKM93}*{Chapter~2}, \cite{MZ97}*{Section~2.1} and~\cite{M11}*{Section~2.2}.
Here and in the following, we repeatedly use the identification $S^{\alpha,p}(\R^n)=L^{\alpha,p}(\R^n)$ for $\alpha\in(0,1)$ and $p\in(1,+\infty)$ proved in~\cite{BCCS20}*{Corollary~2.1}.  

\begin{definition}[The $(\alpha,p)$-capacity]
Let $\alpha\in(0,1)$ and $p\in[1,+\infty)$.
We let 
\begin{equation*}
\Capa_{\alpha,p}(K)
=
\inf
\set*{\|f\|_{S^{\alpha,p}(\R^n)}^p : f\in C^\infty_c(\R^n),\ f\ge\chi_K}
\end{equation*}
be the \emph{$(\alpha,p)$-capacity} of the compact set $K\subset\R^n$.
\end{definition}

The mapping $\Capa_{\alpha,p}$ can be extended to more general sets via the following standard routine. 
If $A\subset\R^n$ is an open set, then we set 
\begin{equation*}
\Capa_{\alpha,p}(A)
=
\sup\set*{\Capa_{\alpha,p}(K) : K\subset A,\ K\ \text{compact}}
\end{equation*}
and so, given any set $E\subset\R^n$, we let
\begin{equation*}
\Capa_{\alpha,p}(E)
=
\sup\set*{\Capa_{\alpha,p}(A) : A\supset E,\ A\ \text{open}}.
\end{equation*}

We now recall the notion of \emph{$(\alpha,p)$-quasievery point}, see~\cite{AH96}*{Definition~2.2.5}.

\begin{definition}[$(\alpha,p)$-quasievery point]
Let $\alpha\in(0,1)$ and $p\in[1,+\infty)$.
We say that a property $\mathscr P(x)$ is true for \emph{$(\alpha,p)$-quasievery $x\in\R^n$} if
\begin{equation*}
\Capa_{\alpha,p}(\set*{x\in\R^n : \mathscr P(x)\ \text{is false}})=0.
\end{equation*}
\end{definition}

Recall that, if $\alpha\in(0,1)$ and $p\in\left(\frac n\alpha,+\infty\right)$, then $S^{\alpha,p}(\R^n)\subset C_b(\R^n)$ continuously by the fractional Sobolev Embedding Theorem, see~\cite{AH96}*{Theorem~1.2.4(c)} for instance.
For this reason, the notion of $(\alpha,p)$-capacity becomes interesting only when $\alpha p\le n$ (see the discussion below~\cite{AH96}*{Proposition~2.6.1}). 
Precisely, if $\alpha\in(0,1)$ and $p\in\left(1,\frac n\alpha\right]$, then
$
\Haus{n-\alpha p+\eps}
\ll
\Capa_{\alpha,p}
$
for all $\eps>0$, see~\cite{AH96}*{Theorem~5.1.13 and Corollary~5.1.14}.
 
\subsection{The precise representative}

We now study the precise representatives of $BV^{\alpha,p}$ functions by combining the embedding proved in \cref{res:BV_alpha_p_in_S_beta_q} with the results already known in the literature for the precise representatives of functions in fractional Bessel potential spaces.

We begin by recalling the definition of \emph{quasicontinuity}.
For the integer case $\alpha=1$, we refer the reader to~\cite{AH96}*{Definition~6.1.1} and~\cite{EG15}*{Definition~4.11}.

\begin{definition}[$(\alpha,p)$-quasicontinuity]
We say that a function $f\colon\R^n\to[-\infty,+\infty]$ defined $(\alpha,p)$-quasieverywhere is \emph{$(\alpha,p)$-quasicontinuous} if, for each $\eps>0$, there exists an open set $A_\eps\subset\R^n$ such that 
$\Capa_{\alpha,p}(A_\eps)<\eps$ and
$f|_{\R^n\setminus A_\eps}$ is continuous.
\end{definition}

Here and in the following, the \emph{precise representative} of a function $u\in L^1_{\loc}(\R^n;\R^m)$ is defined as
\begin{equation*}
u^\star(x)
=
\lim_{r\to0^+}\aint_{B_r(x)} u(y)\,dy,
\quad
x\in\R^n,
\end{equation*}
if the limit exists, otherwise $u^\star(x)=0$ by convention.
The following result provides a precise description of the continuity properties of the precise representative of a function in $S^{\alpha,p}(\R^n)$ for $p\in\left(1,\frac n\alpha\right]$. 
We refer the reader to~\cite{AH96}*{Theorem~6.2.1} for the proof. 

\begin{theorem}[Quasicontinuity of $S^{\alpha,p}$ functions]
\label{res:quasicont_S_alpha_p}
Let $\alpha\in(0,1)$ and 
$p\in\left(1,\frac n\alpha\right]$. 
If $f\in S^{\alpha,p}(\R^n)$, then~$f^\star$ is an $(\alpha,p)$-quasicontinuous representative of~$f$ and 
\begin{equation*}
\lim_{r\to0^+}\aint_{B_r(x)} |f(y)-f^\star(x)|^q\,dy
=0
\end{equation*}
for $(\alpha,p)$-quasievery $x\in\R^n$, where
\begin{equation*}
q\in
\begin{cases}
\left[1,\frac{np}{n-\alpha p}\right]
& \text{if}\ \alpha p<n,\\[3mm]
[1,+\infty)
& \text{if}\ \alpha p=n.
\end{cases}
\end{equation*}
\end{theorem}

Thanks to the embedding proved in \cref{res:BV_alpha_p_in_S_beta_q}, we immediately deduce the following result concerning the quasicontinuity of the functions in $BV^{\alpha,p}(\R^n)$.

\begin{corollary}[Quasicontinuity of $BV^{\alpha,p}$ functions for $p<\frac n{n-\alpha}$] \label{res:quasicont_p_special}
Let $\alpha,\beta\in(0,1)$, with $\beta<\alpha$, and let $p, q \in[1,+\infty]$ be such that $p\le q <\frac n{n+\beta-\alpha}$, with $q>1$.
If $f\in BV^{\alpha,p}(\R^n)$, then $f^\star$ is a $(\beta,q)$-quasicontinuous representative of~$f$ and 
\begin{equation*}
\lim_{r\to0^+}\aint_{B_r(x)} |f(y)-f^\star(x)|^{t}\,dy
=0
\end{equation*}	
for $(\beta,q)$-quasievery $x\in\R^n$ and for all $t\in\left[1,\frac{nq}{n-\beta q}\right]$.
\end{corollary}

In order to provide an extension of \cref{res:quasicont_p_special} also for all exponents $p\in[1,+\infty]$, we need the following localization result for $BV^{\alpha,p}$ functions.

\begin{lemma}[Localization for $BV^{\alpha,p}$ functions for $p\in{[}1,+\infty{]}$]
\label{res:localization}
Let $\alpha\in(0,1)$ and let $p\in[1,+\infty]$.
If $f\in BV^{\alpha,p}(\R^n)$ and $\eta\in\Lip_c(\R^n)$, then $f\eta\in BV^{\alpha,q}(\R^n)$ for all $q\in[1,p]$, with
\begin{equation}
\label{eq:localization}
D^\alpha(f\eta)
=
\eta\,D^\alpha f
+
f\,\nabla^\alpha\eta
\,\Leb{n}
+
\nabla^\alpha_{\rm NL}(f,\eta)
\,\Leb{n}
\quad
\text{in}\
\M (\R^n;\R^n).
\end{equation}
\end{lemma}

\begin{proof}
By H\"older's inequality, we clearly have that $f\eta\in L^q(\R^n)$ for all $q\in[1,p]$, so that we only need to prove~\eqref{eq:localization}.
First of all, note that the right-hand side of~\eqref{eq:localization} is well posed because $\eta\in\Lip_c(\R^n)$. 
In particular, we have $\nabla^\alpha\eta\in L^1(\R^n;\R^n)\cap L^\infty(\R^n;\R^n)$ by~\cite{CS19}*{Corollary~2.3} and $\nabla^\alpha_{\rm NL}(f,\eta)\in L^1(\R^n)$, since by Minkowski's and H\"older's (generalized) inequalities we can estimate
\begin{equation}
\label{eq:KS_trick}
\begin{split}
\|\nabla^\alpha_{\rm NL}(f,\eta)\|_{L^1(\R^n;\R^n)}
&\le
\mu_{n,\alpha}
\int_{\R^n}
\frac{\||f(\cdot+h)-f|\,|\eta(\cdot+h)-\eta|\|_{L^1(\R^n)}}{|h|^{n+\alpha}}\,dh
\\
&\le
2\mu_{n,\alpha}
\,
\|f\|_{L^p(\R^n)}
\int_{\R^n}
\frac{\|\eta(\cdot+h)-\eta\|_{L^{p'}(\R^n)}}{|h|^{n+\alpha}}\,dh
\\
&\le
c_{n,\alpha,p}
\,
\|f\|_{L^p(\R^n)}
\,
\|\eta\|_{W^{1,{p'}}(\R^n)}
\end{split}
\end{equation}
(for the validity of the last inequality, see~\cite{Leoni17}*{Theorem~17.33} for instance).
Now let $\phi\in\Lip_c(\R^n;\R^n)$ be given.
By~\cite{CS19}*{Lemma~2.7}, we can write 
\begin{equation*}
\div^\alpha(\eta\phi)
=
\eta\,\div^\alpha\phi
+
\phi\cdot\nabla^\alpha\eta
+\div_{\rm NL}^\alpha(\eta,\phi),
\end{equation*} 
so that
\begin{align*}
\int_{\R^n}f\eta\,\div^\alpha\phi\,dx 
=
\int_{\R^n}f\,
\div^\alpha(\eta\phi)
\,dx
-
\int_{\R^n}
f\phi\cdot\nabla^\alpha\eta
\,dx
-
\int_{\R^n}
f\,\div^\alpha_{\rm NL}(\eta,\phi)\,
dx.
\end{align*}
In addition, since $\eta\phi\in\Lip_c(\R^n;\R^n)$ and $f\in BV^{\alpha,p}(\R^n)$, we immediately see that 
\begin{equation*}
\int_{\R^n}f\,
\div^\alpha(\eta\phi)
\,dx
=
-
\int_{\R^n}\eta\phi\cdot
dD^\alpha f.
\end{equation*}
Finally, let $(f_\eps)_{\eps>0}\subset BV^{\alpha,p}(\R^n)\cap C^\infty(\R^n)$ be such that $f_\eps=f*\rho_\eps$ for all $\eps>0$ as in \cref{res:moll_conv_alpha_p}. 
Note that $f_\eps\in\Lip_b(\R^n;\R^n)$, so that
\begin{equation*}
\int_{\R^n}
f_\eps\,\div^\alpha_{\rm NL}(\eta,\phi)\,
dx
=
\int_{\R^n}
\phi\cdot\nabla^\alpha_{\rm NL}(f_\eps,\eta)\,
dx.
\end{equation*}
for all $\eps>0$ by~\cite{CS19-2}*{Lemmas~2.4 and~2.5}.
Now, arguing as in the proof of~\eqref{eq:KS_trick}, we can infer that 
\begin{equation}
\label{eq:ibp_NL_eps_Lip_b}
\lim_{\eps\to0^+}
\nabla^\alpha_{\rm NL}(f_\eps,\eta)
=
\nabla^\alpha_{\rm NL}(f,\eta)
\quad
\text{in}\ L^1_{\loc}(\R^n;\R^n),
\end{equation}
so that we can pass to the limit as~$\eps\to0^+$ in~\eqref{eq:ibp_NL_eps_Lip_b} to get that
\begin{equation*}
\int_{\R^n}
f\,\div^\alpha_{\rm NL}(g,\phi)\,
dx
=
\int_{\R^n}
\phi\cdot\nabla^\alpha_{\rm NL}(f,g)\,
dx.
\end{equation*}
In conclusion, we have that 
\begin{equation*}
\int_{\R^n}f\eta\,\div^\alpha\phi\,dx 
=
-
\int_{\R^n}\eta\phi\cdot
dD^\alpha f
-
\int_{\R^n}
f\phi\cdot\nabla^\alpha\eta
\,dx
-
\int_{\R^n}
\phi\cdot\nabla^\alpha_{\rm NL}(f,\eta)\,
dx
\end{equation*}
for any given $\phi\in\Lip_c(\R^n;\R^n)$ and the proof is complete.
\end{proof}

\begin{corollary}[Quasicontinuity of $BV^{\alpha,p}$ functions for $p\in{[}1,+\infty{]}$] \label{res:quasicont_any_p}
Let $\alpha,\beta\in(0,1)$ be such that $\beta<\alpha$ and let $p\in\left[1, + \infty \right]$ and $q\in\left[1,\frac{n}{n+\beta-\alpha}\right)$.
If $f\in BV^{\alpha,p}(\R^n)$, then $f^\star$ is a $(\beta,q)$-quasicontinuous representative of~$f$ and 
\begin{equation} \label{eq:quasicont_any_p}
\lim_{r\to0^+}\aint_{B_r(x)} |f(y)-f^\star(x)|^{t}\,dy
=0
\end{equation}	
for $(\beta,q)$-quasievery $x\in\R^n$ and for all $t\in\left[1,\frac{nq}{n-\beta q}\right]$. 
In particular, the precise representative of $f$ is well defined $\Haus{n - \alpha + \eps}$-a.e.\ for all $\eps >0$.
\end{corollary}

\begin{proof}
By \cref{res:localization}, we know that $f\eta \in BV^{\alpha, 1}(\R^n)$ for all $\eta \in \Lip_c(\R^n)$. 
Hence, \cref{res:quasicont_p_special} implies the existence of a $(\beta,q)$-quasicontinuous representative of $f\eta$ for all $\beta \in (0, \alpha)$ and $q\in\left (1,\frac{n}{n+\beta-\alpha}\right)$. 
In particular, if $\eta(x)=1$ for all $x\in B_R$ for some given $R>0$, then we get the existence of $(f\eta)^\star(x) = f^\star(x)$ for $(\beta,q)$-quasievery $x\in\R^n$, together with \eqref{eq:quasicont_any_p}. Since $R > 0$ is arbitrary, $f^\star(x)$ must exist for $(\beta,q)$-quasievery $x\in\R^n$. 
Finally, since $q < \frac{n}{\beta}$, we have
$
\Haus{n-\beta q+\delta}
\ll
\Capa_{\beta,q}
$
for all $\delta>0$.
Thus, by optimizing in $\beta \in (0, \alpha)$ and in $q\in\left (1,\frac{n}{n+\beta-\alpha}\right)$, the existence of $f^\star(x)$ follows for $\Haus{n - \alpha + \eps}$-a.e.\ $x \in \R^n$ and the proof is complete.
\end{proof}

%%% BIBLIOGRAPHY %%%

\end{document}